\documentclass[reqno,10pt]{amsart}
\usepackage{amsmath,amsthm,amsfonts,amssymb}
\usepackage[left=3cm,right=3cm,top=3.5cm,bottom=3.5cm,a4paper]{geometry}

\usepackage{hyperref}
\usepackage[all]{xy}
\usepackage{graphicx}
\usepackage{amscd}
\usepackage{color}
\usepackage{enumerate}
\usepackage{fourier}
\usepackage{tikz}
\usepackage{cases}

\hypersetup{
    colorlinks=true,
    linkcolor=blue,
    citecolor=green,
    urlcolor=cyan
}

\numberwithin{equation}{section}

\theoremstyle{plain} 
\newtheorem{theorem}{Theorem}[section]
\newtheorem{corollary}[theorem]{Corollary}
\newtheorem{lemma}[theorem]{Lemma}

\theoremstyle{definition}
\newtheorem{definition}[theorem]{Definition}

\newtheorem{conjecture}[theorem]{Conjecture}

\newtheorem{notation}[theorem]{Notation}
\newtheorem{assumption}[theorem]{Assumption}

\newtheorem{thmx}{Theorem}

\theoremstyle{remark}
\newtheorem{remark}[theorem]{Remark}

\newcommand{\PP}{\mathbb{P}}

\newcommand{\Z}{\mathbb{Z}}

\newcommand{\C}{\mathbb{C}}

\newcommand{\CA}{\mathcal{A}}

\newcommand{\CF}{\mathcal{F}}

\newcommand{\AI}{{A_\infty}}

\newcommand{\Jac}{{\rm Jac}}

\newcommand{\HG}{{\widehat{G}}}
\newcommand{\id}{\mathrm{id}}
\newcommand{\bL}{{\mathbb{L}}}

\newcommand{\barW}{{\overline{W}}}
\newcommand{\tildeW}{{\widetilde{W}}}
\newcommand{\Dup}{{D^{\uparrow}}}
\newcommand{\Ddown}{{D^{\downarrow}}}
\newcommand{\MC}{{\mathcal{MC}}}

\newcommand{\Hom}{{\rm Hom}}

\newcommand{\nablaxy}[1]{{\nabla_{#1}^{x \to (x,y)}}}
\newcommand{\nablayz}[1]{{\nabla_{#1}^{y \to (y,z)}}}

\allowdisplaybreaks

\begin{document}
\title[Twisted Jacobians as endomorphisms of matrix factorizations]{Twisted Jacobian algebras as endomorphism algebras of \\ equivariant matrix factorizations}

\author{Sangwook Lee}
\address{Sangwook Lee: Department of Mathematics and Integrative Institute of Basic Science \\Soongsil University \\
369 Sangdo-ro, Dongjak-gu, Seoul, Korea}
\email{sangwook@ssu.ac.kr}
\begin{abstract}
For a polynomial $W$ with an isolated singularity, we can consider the Jacobian ring as an invariant of the singularity. If in addition we have a group action on the polynomial ring with $W$ fixed, we are led to consider the twisted Jacobian ring which reflects the equivariant structure as well. Our main result is to show that the twisted Jacobian ring is isomorphic to an endomorphism ring of the "twisted diagonal" matrix factorization. As an application, we suggest a way to investigate Floer theory of Lagrangian submanifolds which represent homological mirror functors.
\end{abstract}
\maketitle

\section{Introduction}

Given a nontrivial automorphism group $G$ of the space-time, it is natural to consider $G$-equivariant 2-dimensional topological field theory (2d TFT for short). Turaev formulated closed $G$-equivariant 2d TFT as a $G$-graded algebra $\mathcal{C}=\bigoplus_{g\in G}\mathcal{C}_g$ (which was later called {\em Turaev algebra}) satisfying a number of axioms which include $G$-twisted commutativity, $G$-invariance of the pairing, trace axiom etc. If we take the invariant of each graded piece $\mathcal{C}_g$ by centralizer $Z_g$, then it gives rise to the orbifold theory (see \cite{Dbrane,Tu}). 

There is a well-known geometric example of Turaev algebras. When $X$ is a compact almost complex manifold and $G$ is a finite group acting on $X$, Fantechi-G\"ottsche \cite{FG} defined an algebra $\mathcal{H}(X,G)$, so-called the {\em stringy cohomology} of $(X,G)$. It was shown that $\mathcal{H}(X,G)$ satisfies axioms of Turaev algebras. It is also known that the $G$-invariant subalgebra $\mathcal{H}(X,G)^G$ is isomorphic to the orbifold cohomology $H_{orb}^*([X/G])$ (see also \cite{Ki}).

We recall another important example of TFTs from singularity theory which is the main topic of this paper. Let $X$ be a variety and $W$ be an algebraic function on $X$. When we want to investigate the singularity of $W$, we are led to consider an important invariant, so-called Jacobian ring (denoted by $\Jac(W)$) which is the function ring of the singular locus of $W$. Together with residue pairing, Jacobian ring is a Frobenius algebra which encodes a (non-equivariant) closed 2d TFT.
If a group $G$ acts on the variety $X$ and the function $W$ is invariant under the action, then it is natural to consider $G$-equivariant invariants of the singularity. If $G$ is finite and abelian, such a pair $(W,G)$ is called a {\em Landau-Ginzburg orbifold} which is an interesting object in mirror symmetry. For simplicity we only consider $(W,G)$ where $W$ is a polynomial in $R=k[x_1,\cdots,x_n]$ ($k$ is a field of characteristic $0$) which has a singularity at the origin, and $G$ is a finite abelian group which acts diagonally on $R$.
It is natural to ask how we can construct an algebra structure which corresponds to the closed $G$-equivariant 2d TFT for $(W,G)$. For convenience, let us introduce its name as {\em twisted Jacobian algebra} of $(W,G)$ and denote it by $\Jac'(W,G)$. There is a natural candidate for the underlying module structure of $\Jac'(W,G)$ as
\begin{equation}\label{eq:underlyingmodule}
 \Jac'(W,G)=\bigoplus_{g\in G} \Jac(W^g)\cdot \xi_g
 \end{equation}
 where $W^g:=W|_{{\rm Fix}(g)}$ and $\xi_g$ is a formal generator with $|\xi_g|=n-\dim({{\rm Fix}(g)})$.

We recall earlier works on the construction of twisted Jacobian algebras.
\begin{itemize}
\item Kaufmann \cite{Kau} classified possible algebraic structures on \eqref{eq:underlyingmodule} in terms of discrete torsions of $G$. He also observed that one needs to allow nontrivial characters which modify Turaev's axioms to accommodate more general classes of singularities.
\item Basalaev-Takahashi-Werner \cite{BTW} investigated $\Jac'(W,G)$ in the context of mirror symmetry of Frobenius manifolds which essentially involve primitive forms. They axiomatized twisted Jacobian algebras employing primitive forms, and obtained existence and uniqueness result for twisted Jacobian algebras of invertible polynomials together with diagonal symmetry group actions. 
\item Shklyarov \cite{Shk} gave an explicit multiplication formula of $\Jac'(W,G)$ for arbitrary Landau-Ginzburg orbifolds. He deduced the formula from the cup product on Hochschild cohomology of a curved algebra with crossed product coefficient. 
\item He-Li-Li \cite{HLL} deduced a combinorial multiplication formula together with verification of Kaufmann's axioms (concerning not only multiplications, but also pairings and trace axiom etc) for invertible polynomials.
\end{itemize}

In this paper, we will focus on the product structure on $\Jac'(W,G)$.

Recall that taking invariant of a closed $G$-equivariant 2d TFT by centralizers gives rise to the orbifold theory. It is also the case for twisted Jacobian algebras. Namely, if we take $G$-invariant of $\Jac'(W,G)$, it becomes a commutative Frobenius algebra $\Jac(W,G)$ which is called the {\em orbifold Jacobian algebra} of $(W,G)$. Orbifold Jacobian algebras naturally arise from categories as follows.
\begin{theorem}[\cite{PP,Shk}]
We have an algebra isomorphism
\begin{equation}\label{thm:jachh} 
\Jac(W,G) \cong HH^*(MF_G(W))
\end{equation}
where $HH^*(MF_G(W))$ is the Hochschild cohomology of the dg category of $G$-equivariant matrix factorizations of $W$.
\end{theorem}

Computation of Hochschild cohomology of a dg category is in general not easy. For a dg category of matrix factorizations, however, we can realize the Hochschild cochain complex very explicitly. The original definition (due to To\"en \cite{To}) of Hochschild cochain complex of a dg category $\CA$ is
\[ C^*(\CA,\CA):= hom_{R\underline{Hom}_c(\CA,\CA)}(\mathrm{Id}_\CA,\mathrm{Id}_\CA),\]
where $R\underline{Hom}_c(\CA,\CA)$ is the dg category of continuous endofunctors of $\CA$. 

The notion of {\em kernels} for dg functors between categories of matrix factorizations enables us to compute Hochschild cohomology. The theory of kernels for functors between (possibly non-affine) matrix factorization categories was extensively studied in \cite{BFK}. For affine case, we have the following results.
\begin{itemize}
\item If $\CA=MF(W)$, then $R\underline{Hom}_c(\CA,\CA) \simeq MF(W\boxminus W)$, and the identity functor $\mathrm{Id}_\CA$ corresponds to a Koszul matrix factorization $\Delta_W$ of the diagonal ideal. 
\item If $\CA=MF_G(W)$, then $R\underline{Hom}_c(\CA,\CA) \simeq MF_{G\times G}(W\boxminus W)$ and the identity functor $\mathrm{Id}_\CA$ corresponds to a $(G\times G)$-equivariant Koszul matrix factorization $\Delta_W^{G\times G}$ of the {\em twisted} diagonal ideal.
\end{itemize}
Therefore we can reformulate \eqref{thm:jachh} as follows.
\begin{equation}\label{eq:jacendo} 
\Jac(W,G) \cong \Hom_{MF_{G\times G}(W\boxminus W)}(\Delta_W^{G\times G},\Delta_W^{G\times G}).
\end{equation}

The isomorphism \eqref{eq:jacendo} gives a clue to interpret $\Jac'(W,G)$ in terms of matrix factorizations. 
Our approach towards multiplication formula of $\Jac'(W,G)$ is via category $MF_{1\times G}(W\boxminus W)$. Such a choice of the category is natural in the sense that $\Jac'(W,G)$ can be thought of a $G$-lifting of $\Jac(W,G)$. The following is our main theorem.

\begin{thmx}\label{maintheorem}
For a Landau-Ginzburg orbifold $(W,G)$, we have an algebra isomorphism
\begin{equation}\label{eq:mainthm} 
\Jac'(W,G) \cong \Hom_{MF_{1\times G}(W\boxminus W)} (\Delta_W^{G\times G},\Delta_W^{G\times G}).
\end{equation}
\end{thmx}

Shklyarov's multiplication formula comes from Hochschild cohomology of a curved algebra which is the endomorphism algebra of a compact generator of the category. On the other hand, we obtain the formula inside the whole category of matrix factorizations. 
Since we take an explicit Koszul matrix factorization and consider its endomorphisms, the product is described by calculus of Clifford generators which are nothing but differential forms and interior products. 
In our formulation, complicatedness appears in the expression of closed endomorphisms of $\Delta_W^{G\times G}$. 
Therefore, classification of closed endomorphisms of $\Delta_W^{G\times G}$ is the most important starting point. 
We will observe that we recover Shklyarov's formula by restricting our multiplication of endomorphisms to the twisted diagonal.

The main motivation of this work is based on the mirror symmetry. Singularity often appears as a mirror counterpart of a symplectic manifold. Roughly, a Lagrangian submanifold $\bL\subset X$ which can be "deformed" gives rise to an algebraic function $W_\bL$ on the deformation space. In some case, the singularity theory of $W_\bL$ is "equivalent" to the symplectic geometry of $X$ itself, in which case we say that $W_\bL$ is a {\em Landau-Ginzburg mirror} to $X$. A concrete mirror symmetry statement was given by Kontsevich as follows.
\begin{conjecture}[Homological Mirror Symmetry]
Let $X$ be a symplectic manifold and $W: \check{X} \to k$ be a Landau-Ginzburg mirror to $X$. Then we have an equivalence of categories
\begin{equation}\label{eq:HMS}
 Fuk(X) \simeq MF(W).
 \end{equation}
\end{conjecture}
For some symplectic manifolds, we can construct an explicit (curved) Yoneda functor $\CF^\bL$ giving the equivalence \eqref{eq:HMS} (see \cite{CHL1,CHL2} for example) with the representing Lagrangian submanifold $\bL$. In \cite{CL} it was shown that Lagrangian Floer theory of $\bL$ is deeply related to $\Jac(W_\bL)$. We consider a technical assumption (which is Assumption \ref{as:floergen}) which is expected to be satisfied by representing Lagrangians for mirror functors, and is indeed true for various examples. If $\bL\subset X$ satisfies Assumption \ref{as:floergen}, then the following holds:
\[ HF(\bL^x,\bL^x)\otimes R \cong \Jac(W_\bL)\]
where $R$ means the function ring of the deformation space of $\bL$ and $x$ means its coordinates. We refer to \cite{CL} for details and examples.

The story becomes even more interesting for $\bar{\bL} \subset X/G$, where $X/G$ is a global quotient orbifold and $\bar{\bL}$ satisfies Assumption \ref{as:floergen}. Let $\pi: X \to X/G$ be the quotient map. It was shown in \cite{CL} that their $\HG$-invariant subalgebras are indeed isomorphic as algebras.
\begin{equation}\label{eq:orbjacresult}
 \big( HF\big(\pi^{-1}(\bar{\bL}^x),\pi^{-1}(\bar{\bL}^x)\big)\otimes R\big)^\HG \cong \Jac'(W_{\bar{\bL}},\HG)^\HG.
\end{equation}

Our application of Theorem \ref{maintheorem} is that we indeed have an isomorphism of "$\HG$-liftings" of above.
\begin{thmx} We have an algebra isomorphism
\begin{equation}\label{eq:mainquestion}
 HF\big(\pi^{-1}(\bar{\bL}^x),\pi^{-1}(\bar{\bL}^x)\big)\otimes R \cong \Jac'(W_{\bar{\bL}},\HG).
 \end{equation}
\end{thmx}

\begin{proof}
We combine the following algebra isomorphism in \cite{CL} with \eqref{eq:mainthm}:
\begin{equation*}
 HF\big(\pi^{-1}(\bar{\bL}^x),\pi^{-1}(\bar{\bL}^x)\big)\otimes R \cong \Hom_{MF_{1\times \HG}(W_{\bar{\bL}}\boxminus W_{\bar{\bL}})}(\Delta_{W_{\bar{\bL}}}^{\HG\times\HG},\Delta_{W_{\bar{\bL}}}^{\HG\times\HG}).\qedhere
\end{equation*}
\end{proof}

We conclude the introduction by the following nontrivial result on the structure of Floer cohomology of $\pi^{-1}(\bar{\bL}^x)$.

\begin{thmx}
$HF\big(\pi^{-1}(\bar{\bL}^x),\pi^{-1}(\bar{\bL}^x)\big)\otimes R$ is a $\HG$-twisted commutative algebra.
\end{thmx}


The paper is organized as follows. In Section \ref{sec:prelim}, we recall the notion of twisted Jacobian algebras following \cite{Shk}. Then we give a brief overview on categories of equivariant matrix factorizations. In particular, we investigate the structure of a kernel $\Delta_W^{G\times G}$ for $\mathrm{Id}_{MF_G(W)}$ in detail. We conclude Section \ref{sec:prelim} by showing that if we consider $(1\times G)$-equivariant endomorphisms of $\Delta_W^{G\times G}$, then it gives rise to a module which is isomorphic to $\Jac'(W,G)$, using spectral sequence argument. In Section \ref{sec:classification}, we classify all closed $(1\times G)$-equivariant endomorphisms of $\Delta_W^{G\times G}$. We remark that the endomorphism space of $\Delta_W^{G\times G}$ is a double complex and we find its closed elements by "extending leading order terms". In Section \ref{sec:comparemulti} we compare multiplications of our closed endomorphisms of $\Delta_W^{G\times G}$ and those in \cite{Shk}. In Section \ref{sec:application} we connect our main result to Lagrangian Floer theory and investigate an example of Floer cohomology in $T^2$ with $\Z_6$-action.

\subsection{Standing assumptions and notations}
\begin{enumerate}
\item $R=k[x_1,\cdots,x_n]$ is a polynomial ring over a field $k$. $W$ is a polynomial in $R$ and has a critical point at the origin with the critical value $0$. $G$ is a finite abelian group which acts diagonally on $k[x_1,\cdots,x_n]$, i.e. $g\cdot x_i=g_i x_i$ for some $g_i \in k^*$. We also let $S=k[x_1,\cdots,x_n,y_1,\cdots,y_n]$.
\item $W$ is invariant under $G$-action, i.e.
\[ W(gx_1,\cdots,gx_n)=W(x_1,\cdots,x_n)\] for any $g\in G$. 
\item For $g\in G$, define
\begin{equation}\label{eq:movindex}
 I^g:=\{i \mid 1\leq i \leq n, \; g\cdot x_i=x_i\},\quad I_g:= \{1,\cdots,n\}\setminus  I^g, \quad d_g:=| I_g|.
 \end{equation}
Denote the natural projection map by
\[ \pi_g: R \to R/(x_i : i\in I_g),\]
and for any $f\in R$,
\begin{equation}\label{def:Wfixed}
 f^g:=\pi_g(f).
 \end{equation}
 In particular, 
 \[ x_i^g=
 \begin{cases}
 x_i&{\rm if\;\;} gx_i=x_i \\
 0& {\rm otherwise.}
 \end{cases}
 \]
\item Sometimes we just write
\[ W(x):=W(x_1,\cdots,x_n),\quad W(y):=W(y_1,\cdots,y_n).\]
Also, for a polynomial $f \in S$, we write 
\[ f(x,y):=f(x_1,\cdots,x_n,y_1,\cdots,y_n).\]
We will specify variables for $W\boxminus W$ as $W(y)-W(x)$.
\end{enumerate}
 \subsection{Acknowledgements}
The author would like to express his deep gratitude to the late Bumsig Kim. His life, not only as a mathematician but also as a person, was always an inspiration. The author has been greatly indebted to him for his continuous encouragement, support and the brilliant insight which he willingly shared with his companions. The author would like to thank Cheol-Hyun Cho for his interest in this work and giving a number of suggestions on the earlier draft. 

This work was supported by the National Research Foundation of Korea(NRF) grant funded by the Korea government(MSIT)(No. 202117221032) and Basic Science Research Program through the National Research Foundation of Korea (NRF) funded by the Ministry of Education (2021R1A6A1A10044154).

\section{Preliminaries}\label{sec:prelim}
In this section, first we briefly recall the construction of twisted Jacobian algebras mainly following \cite{Shk}. Then we collect some essential ingredients concerning equivariant matrix factorizations. Finally we reconstruct the underlying module of a twisted Jacobian algebra in terms of matrix factorizations. 
\subsection{Twisted Jacobian algebras}
\begin{definition}
The {\em Jacobian algebra} of $W\in k[x_1,\cdots,x_n]$ is defined by
\[ \Jac(W):= \frac{k[x_1,\cdots,x_n]}{(\frac{\partial W}{\partial x_1}, \cdots, \frac{\partial W}{\partial x_n})}.\]
\end{definition}
Before introducing $G$-twisted Jacobian algebras of Landau-Ginzburg orbifolds, we need to recall the following.
\begin{definition} The {\em $n$th graded Clifford algebra} is a $\Z$-graded $k$-algebra 
\[ \mathrm{Cl}_n:=k[\theta_1,\cdots,\theta_n,\partial_{1},\cdots,\partial_{n}],\;\; |\theta_i|=-1,\;\; |\partial_{i}|=1 \;\;\text{for all $i$}\]
together with relations
\begin{equation}\label{eq:cliffrel}
 \theta_i \theta_j=-\theta_j \theta_i,\;\; \partial_{i}\partial_{j}=-\partial_{j}\partial_{i},\;\; \partial_{i}\theta_j=-\theta_j\partial_{i}+\delta_{ij}.
 \end{equation}
\end{definition}

\begin{notation}
For an ordered subset $I=\{i_1<\cdots<i_r\} \subset \{1,\cdots,n\}$, we denote
\[ \theta_I:=\theta_{i_1}\cdots\theta_{i_r},\quad \partial_I:=\partial_{i_1}\cdots\partial_{i_r}.\]
\end{notation}
Let $W$ be a polynomial in variables $x=(x_1,\cdots,x_n)$. We define a difference operator
\[ \nabla_j^{x \to (x,y)} W:= \frac{W(y_1,\cdots,y_j,x_{j+1},\cdots,x_n)-W(y_1,\cdots,y_{j-1},x_j,\cdots,x_n)}{y_j-x_j}.\]
A successive application of difference operators like $\nabla_i^{y \to (y,z)}\nabla_j^{x \to (x,y)}W$ for $i \leq j$ means that we obtain a polynomial in variables $(x,y)$ first by $\nabla_j^{x\to (x,y)}$, then we treat it as a polynomial only in $y$-variables(that is, consider $x$-variables as constants) and apply $\nabla_i^{y \to (y,z)}$.

\begin{definition}\label{def:ojr}[\cite{Shk}]
The {\em $G$-twisted Jacobian algebra} of a Landau-Ginzburg orbifold $(W,G)$ is a $(\Z_2\times G)$-graded algebra
\[\Jac'(W,G) = \bigoplus_{h \in G} \Jac(W^h) \cdot \xi_h \]
where $\xi_h$ is a formal generator of degree $d_h \; {\rm mod}\; 2$. The structure constant $\sigma_{g,h}\in \Jac(W^{gh})$ for the product of $\xi_g$ and $\xi_h$ is given by the coefficient of ${\theta_{I_{gh}}}:=\prod_{i \in I_{gh}} {\theta_i}$ in the following expression
\begin{equation}\label{eq:twjacprod}
\frac{1}{d_{g,h}!} \Upsilon\big( (\lfloor H_W(x,g\cdot x,x)\rfloor_{gh}+\lfloor H_{W,g}(x)\rfloor_{gh}\otimes 1 + 1\otimes \lfloor H_{W,h}(g\cdot x)\rfloor_{gh})^{d_{g,h}}\otimes {\theta_{I_g}} \otimes {\theta_{I_h}}\big),
\end{equation}
where
\begin{itemize}
\item $\displaystyle H_W(x,y,z):=\sum_{1\leq i\leq j \leq n} \nabla_i^{y \to (y,z)}\nabla_j^{x\to (x,y)}(W)\partial_j\otimes\partial_i$.
\item $\displaystyle H_{W,g}:= \sum_{i,j\in I_g,i<j}\frac{1}{1-g_i}\nabla_i^{x \to (x,x^g)}\nabla_j^{x \to (x,g\cdot x)}(W)\partial_i \partial_j \in R[\partial_1,\cdots,\partial_n]$. Recall that 
\[ x_i^g=
\begin{cases}
x_i & {\rm if\;} i\in I^g \\
0& {\rm if\;} i\in I_g.
\end{cases}
\]
 The operations $\nabla_i^{x \to (x,g\cdot x)}$ and $\nabla_i^{x \to (x,x^g)}$ are computed from $\nabla_i^{x \to (x,y)}$ followed by substitutions $y=g\cdot x$ and $y=x^g$, respectively.
\item $\lfloor f \rfloor_{g}:= [f^g] \in \Jac(W^g).$
\item $d_{g,h}:=\frac{d_g+d_h-d_{gh}}{2}$. Define $\sigma_{g,h}=0$ if $d_{g,h}$ is not an integer.
\item The map $\Upsilon$ is defined as
\[\Upsilon:R[\partial_1,\cdots,\partial_n]^{\otimes 2}\otimes R[{\theta_1},\cdots,{\theta_n}]^{\otimes 2} \to R[{\theta_1},\cdots,{\theta_n}],\] 
\[ p_1({\partial})\otimes p_2({\partial})\otimes q_1({\theta})\otimes q_2({\theta}) \mapsto (-1)^{|q_1||p_2|}p_1(q_1)\cdot p_2(q_2),\]
where $p_i (q_i)$ means the action of $p_i(\partial)$ on $q_i(\theta)$ on $k[\theta_1,\cdots,\theta_n]$ via \eqref{eq:cliffrel}.
\end{itemize}
\end{definition}

The algebra in Definition \ref{def:ojr} is worth being called the twisted Jacobian algebra due to the following result.
\begin{theorem}[\cite{Shk}]\label{thm:shkcomm}
$\Jac'(W,G)$ is braided super-commutative, that is
\[ \xi_g \bullet \xi_h=(-1)^{|\xi_g||\xi_h|}\xi_h \bullet h^{-1}\xi_g.\]
\end{theorem}

\subsection{Categories of matrix factorizations}
We recall categorical invariants of singularities.

\begin{definition}
\begin{itemize}
\item A {\em matrix factorization} $(P,d)$ of $W$ is a $\Z/2$-graded free $R$-module $P=P^0 \oplus P^1$ and a degree $1$ endomorphism $d$ of $P$ such that $d^2=W \cdot \id$. 
\item For $j\in \Z/2$, a {\em morphism of degree $j$} from $(P,d)$ to $(Q,d')$ is an $R$-linear map $\phi$ of degree $j$ from $P$ to $Q$. The {\em differential} of $\phi$ is defined by
\[ D\phi:= d' \circ \phi + (-1)^{|\phi|+1}\phi \circ d.\]
\end{itemize}
\end{definition}
It is a standard fact that matrix factorizations of $W$ form a dg category $MF(W)$. Consequently, we can construct the cohomology category $HMF(W)$ from $MF(W)$. Given objects $(P,d)$ and $(Q,d')$, the space of morphisms from $P$ to $Q$ in $MF(W)$ will be denoted by $hom_{MF(W)}(P,Q)$. The morphism space in the cohomology category $HMF(W)$ is denoted by $\Hom_{HMF(W)}(P,Q)$. If the context is clear, we will just write $hom(P,Q)$ and $\Hom(P,Q)$ respectively.

\begin{definition}
Let $G$ be a group which acts on $k[x_1,\cdots,x_n]$ and let $W$ be a polynomial in $k[x_1,\cdots,x_n]$ which is invariant under $G$-action. 
\begin{itemize}
\item A {\em $G$-equivariant matrix factorization} of $W$ is a matrix factorization $(P,d)$ where $P$ is equipped with a $G$-action and $d$ is $G$-equivariant. 
\item The {\em category of $G$-equivariant matrix factorization} $MF_G(W)$ is a dg category which consists of $G$-equivariant matrix factorizations of $W$ as objects, and $G$-equivariant $R$-linear maps between them as morphisms. Differentials and compositions are defined as in nonequivariant case.
\end{itemize}
\end{definition}

Let us review a little bit of To\"en's derived Morita theory and apply it to matrix factorizations.
\begin{definition}[\cite{To}]
Let $\CA$ and $\CA'$ be dg categories. A dg functor $\CF: \CA \to \CA'$ is {\em continuous} if it commutes with direct sums. The dg category of continuous dg functors from $\CA$ to $\CA'$ is denoted by $R\underline{hom}_c(\CA,\CA')$.
\end{definition}

\begin{definition}[\cite{To}]
The {\em Hochschild cochain complex} of a dg category $\CA$ is 
\[ C^*(\CA,\CA):= hom_{R\underline{Hom}_c(\CA,\CA)}(\mathrm{Id}_\CA,\mathrm{Id}_\CA)\]
and its cohomology is called the {\em Hochschild cohomology} of $\CA$, denoted by $HH^*(\CA)$.
\end{definition}
By modifying results in \cite{To} for differential $\Z/2$-graded categories and applying them to matrix factorizations, Dyckerhoff showed the following.
\begin{theorem}[\cite{Dyc}]
There is a quasi-equivalence between two dg categories
\begin{equation}\label{eq:ker}
 R\underline{Hom}_c(MF(R,W),MF(R',W')) \simeq MF(R\otimes_k R', -W\otimes 1+ 1\otimes W').
 \end{equation}
\end{theorem}
 \begin{definition}
For $\CF \in  R\underline{Hom}_c(MF(R,W),MF(R',W'))$, 
 the corresponding matrix factorization under the quasi-equivalence \eqref{eq:ker} is called a {\em kernel} for $\CF$.
\end{definition}

We focus on a specific type of matrix factorizations. It will be used to describe kernels for identity functors of $MF(W)$ and $MF_G(W)$. 
\begin{definition}
Let $W=\sum_{i=1}^m a_i b_i$ for $a_i,b_i \in R$. Let
\[\vec{a}:=(a_1,\cdots,a_m),\vec{b}:=(b_1,\cdots,b_m).\] 
Then the following matrix factorization
\begin{equation}\label{eqK}
\big( R\otimes k[\theta_1,\cdots,\theta_m],\sum_{i=1}^m (a_i \theta_i+b_i \partial_{i}) \big)
\end{equation}
is called the {\em Koszul matrix factorization} of $W$ with respect to $(\vec{a}\mid\vec{b})$.
\end{definition}
Let 
\[ W(y)-W(x):= W(y_1,\cdots,y_n)-W(x_1,\cdots,x_n) \in S.\]
Denote $\nabla_i W:=\nabla_i^{x\to (x,y)}W$ for simplicity. We have
\[ W(y)-W(x)=\sum_{i=1}^n (y_i-x_i)\nabla_i W,\]
so we can construct a Koszul matrix factorization $\Delta_W$ of $W(y)-W(x)$ as follows.
\[ \Delta_W=\Big(S[\theta_1,\cdots,\theta_n], \sum_{i=1}^n \big((y_i-x_i)\theta_i+\nabla_i W(x,y) \cdot\partial_{i}\big)\Big).\]
It is easy to observe that $\Delta_W$ is a kernel for $\mathrm{Id}_{MF(W)}$, because it is given by a resolution of the diagonal ideal $R\subset S$. For the following result see \cite{Dyc} for example.

\begin{theorem}
We have algebra isomorphisms
\[  HH^*(MF(W)) \cong \Hom_{MF(W(y)-W(x))}(\Delta_W,\Delta_W) \cong \Jac(W).\]
\end{theorem}

We turn our attention to $G$-equivariant matrix factorizations. Let $G$ act on $k[x_1,\cdots,x_n]$ diagonally, which means that $h\cdot x_i=h_i x_i$ for $h_i \in k^*$. We define a $G$-action on $k[\theta_1,\cdots,\theta_n,\partial_{\theta_1},\cdots,\partial_{\theta_n}]$ and distinguish it by following notations.
\[ \rho(h)\theta_i:=h_i^{-1} \theta_i, \;\; \rho(h)\partial_{\theta_i}:=h_i \partial_{\theta_i}\]
Since the action is diagonal, $\Delta_W$ is automatically $G$-equivariant.

\begin{theorem}[\cite{PV}]
The following is a $(G\times G)$-equivariant matrix factorization of $W(y)-W(x)$ and is a kernel for ${\rm Id}_{MF_G(W)}$:
\[ \Delta_W^{G\times G}:=\bigoplus_{g\in G}\Big( S[\theta_1,\cdots,\theta_n],\sum_{i=1}^n \big((y_i-gx_i)\theta_i+\nabla_i W(gx,y) \cdot\partial_{i}\big)\Big),\]
consequently we have an isomorphism of algebras:
\[ HH^*(MF_G(W)) \cong \Hom_{MF_{G\times G}(W(y)-W(x))}(\Delta_W^{G\times G},\Delta_W^{G\times G}).\]
\end{theorem}

As we saw in the introduction, we have
\[ HH^*(MF_G(W)) \cong \Jac(W,G).\]
To recover $\Jac'(W,G)$ from matrix factorizations, we consider a category $MF_{1\times G}(W(y)-W(x)).$ We denote each summand of $\Delta_W^{G\times G}$ by following.
\[\Delta_g:=\Big( S[\theta_1,\cdots,\theta_n],\sum\limits_{i=1}^n \big((y_i-gx_i)\theta_i+\nabla_i W(gx,y) \cdot\partial_{i}\big)\Big).\]  
Let $r(x,y)\in S, v\in k[\theta_1,\cdots,\theta_n]$ so that $r(x,y)v\in S[\theta_1,\cdots,\theta_n]$.
Then $(G\times G)$-action on $\Delta_W^{G\times G}$ is given by
\begin{equation}\label{def:ggaction}
 (h_1 \times h_2)\cdot (r(x,y)v):=r(h_1 x,h_2 y)\rho(h_2)v \in \Delta_{h_1 g h_2^{-1}}.
 \end{equation}
\begin{notation}
Let $v=\sum\limits_I c_I \theta_I \in k[\theta_1,\cdots,\theta_n]$. We specify the element $v\in \Delta_g$ by $v_g$. With this notation, \eqref{def:ggaction} can be written as
\[  (h_1 \times h_2)\cdot (r(x,y)v_g)=r(h_1 x,h_2 y)\rho(h_2)v_{h_1 g h_2^{-1}}.\]
\end{notation}


\begin{lemma}\label{lem:1gequivendo}
A $(1\times G)$-equivariant endomorphism of $\Delta_W^{G\times G}$ is completely determined by its restriction to $\Delta_1$.
\end{lemma}

\begin{proof}
If $\phi$ is a $(1\times G)$-equivariant endormorphism, then we have
\[ (1\times h)\cdot \phi\big(r(x,y)v_h\big) = \phi\big( (1\times h) \cdot (r(x,y)v_h)\big) = \phi\big( r(x,hy)\rho(h)v_1\big).\]
Taking $(1\times h^{-1})$-action on both sides, we get
\begin{equation}\label{eq:translation}
 \phi\big(r(x,y)v_h\big) = (1\times h^{-1}) \cdot \big(r(x,hy)\phi(\rho(h)v_1)\big)=r(x,y) \cdot (1\times h^{-1})\cdot\big(\phi(\rho(h)v_1)\big).
 \end{equation}
The right hand side is now completely determined by $\phi|_{\Delta_1}$.
\end{proof}

\begin{corollary}
We have a quasi-isomorphism
\[ hom_{MF_{1\times G}(W(y)-W(x))}(\Delta_W^{G\times G},\Delta_W^{G\times G}) \simeq \bigoplus_{g\in G}hom_{MF(W(y)-W(x))}(\Delta_1,\Delta_g). \]
\end{corollary}
This corollary is really useful as the following lemma indicates.
\begin{lemma} \label{lemma:homjac}
$\Hom_{MF(W(y)-W(x))}(\Delta_1,\Delta_h)$ is a rank $1$ free module over $\Jac(W^h)$.
\end{lemma}

\begin{proof}
Recall that for $h\in G$,
\[\Delta_h=\big( S[\theta_1,\cdots,\theta_n],d(hx,y)\big)\] 
where
\[d(hx,y)=\sum_{i=1}^n \big( (y_i-hx_i)\theta_i + \nabla_i W(hx,y) \partial_i\big).\]
$\Delta_h$ is constructed by the resolution of a shifted MCM module $S/(y_1-hx_1,\cdots,y_n-hx_n)[-n]$ over the hypersurface ring $S/(W(y)-W(x))$. Hence we have
\begin{align}
hom_{MF(W(y)-W(x))}(\Delta_1,\Delta_h) & \simeq hom_S\big(\Delta_1,S/(y_1-hx_1,\cdots,y_n-hx_n)[-n] \big)\nonumber\\
&\simeq hom_S\big(\Delta_1[n],S/(y_1-hx_1,\cdots,y_n-hx_n) \big)\nonumber\\
 &\simeq \Delta_1^\vee[-n] \otimes_S \big(S/(y_1-hx_1,\cdots,y_n-hx_n)\big) \label{homcomplex}
 \end{align}
where the first quasi-isomorphism is due to \cite[Lemma 4.2]{Dyc}.
By self-duality of Koszul matrix factorizations, $\Delta_1^\vee[-n]$ is quasi-isomorphic to $\Delta_1$. Hence we have
\[ \eqref{homcomplex} \simeq \big( R[\theta_1,\cdots,\theta_n],d(x,hx) \big).\]
The latter is the following double complex: 
\[\xymatrixcolsep{0.8pc}\xymatrix{
& && & \vdots && & \vdots & \vdots & \\
& &&& 0 \ar[r] & R\ar[u]_{d_{vert}} \ar[r]^{d_{hor}} & \cdots \ar[r]^-{d_{hor}} & \displaystyle \bigoplus_{1\leq i_1<\cdots<i_{n-1}\leq n}R\cdot \theta_{\{i_1,\cdots,i_{n-1}\}}\ar[u]_{d_{vert}} \ar[r]^-{d_{hor}} & R\cdot \theta_{\{1,\cdots,n\}} \ar[u]_{d_{vert}}\ar[r]  & 0 \\
&& 0\ar[rr] && R \ar[r]^-{d_{hor}}\ar[u] & \displaystyle\bigoplus_{i=1}^n R\cdot \theta_i \ar[u]_{d_{vert}}\ar[r]^-{d_{hor}}&\cdots \ar[r]^-{d_{hor}} & R\cdot \theta_{\{1,\cdots, n\}} \ar[u]_-{d_{vert}}\ar[r] & 0 & \\
0\ar[rr] && R \ar[u]\ar[rr]^-{d_{hor}} && \displaystyle\bigoplus_{i=1}^n R\cdot \theta_i \ar[r]^-{d_{hor}} \ar[u]_-{d_{vert}} & \cdots \ar[r]^-{d_{hor}} & R\cdot \theta_{\{1,\cdots,n\}}\ar[r] &0& &\\
&& \vdots \ar[u]_-{d_{vert}} && \vdots \ar[u]_-{d_{vert}} & & \vdots\ar[u]_-{d_{vert}} & & &
}
\]
where 
\[d_{hor}:=\sum_{i=1}^n (hx_i-x_i)\theta_i,\quad d_{vert}:=\sum_{i=1}^n \nabla_i W(x,hx)\partial_i.\]
We compute its cohomology by spectral sequence from horizontal filtration. The first page is the cohomology with respect to $d_{hor}$. The $i$th row is 
\[\xymatrix{
0 \ar[r] & R\ar[r]^-{d_{hor}} &\bigoplus_{i=1}^n R\cdot \theta_i \ar[r]^-{d_{hor}} & \cdots \ar[r]^-{d_{hor}} &R\cdot \theta_{\{1,\cdots,n\}} \ar[r] & 0,}
\] 
and it is quasi-isomorphic to the following complex:
\[\xymatrix{
0 \ar[r] &  R^h\cdot \theta_{I_h} \ar[r]^-0 & \displaystyle  \bigoplus_{i \in I^h} R^h\cdot\theta_i \theta_{I_h}\ar[r]^-0
&\displaystyle \bigoplus_{\stackrel{i_1<i_2}{ i_1,i_2\in I^h}}R^h \cdot \theta_{i_1}\theta_{i_2}\theta_{I_h}}\]
\[\xymatrixcolsep{1pc}\xymatrix{
\; \ar[r]^-0 & \cdots \ar[r]^-0 &\displaystyle \bigoplus_{\stackrel{i_1<i_2<\cdots<i_{|I^h|-1},}{i_1,\cdots,i_{|I^h|-1}\in I^h}} R^h\cdot \theta_{i_1}\cdots \theta_{i_{|I^h|-1}}\theta_{I_h} \ar[r]^-0 & R^h \theta_{\{1,\cdots,n\}} \ar[r] & 0
}\]
where $R^h:=R/(x_i: i\in I_h)$.

By Shklyarov's Lemma 4.9 in \cite{Shk}, we can replace $d_{vert}$ by $d_{vert}'$ which is defined as
\[d_{vert}':=\sum_{i\in I^h}(\partial_{x_i}W^h)\partial_{\theta_i}.\]
We compute the cohomology of $d_{vert}'$ induced on $E_1$. Since it is isomorphic to the Koszul complex of $(\partial_{x_i}W^h: i\in I^h)$, the cohomology is concentrated on $(|I_h|+i,i)$-th position for all $i\in \Z$. Hence, the spectral sequence converges at $E_2$. Recall that the genuine cocycle of the double complex is constructed by extending tails of a cocycle at $E_2$. Any cocycle at $E_2$ is represented by $\gamma \theta_{I_h}$ for $\gamma \in \Jac(W^h)$, so our assertion holds.
\end{proof}
Combining above results, we have an isomorphism of {\em modules}
\[\Hom_{MF_{1\times G}(W(y)-W(x))}(\Delta_W^{G\times G},\Delta_W^{G\times G})\cong \Jac'(W,G).\] 
To show that it is indeed an algebra isomorphism, we need to classify cocycles in $hom_{MF(W(y)-W(x))}(\Delta_1,\Delta_h)$ for $h\in G$. The following section is devoted to that classification problem.

\section{Classification of $(1\times G)$-equivariant closed endomorphisms of $\Delta_W^{G\times G}$}\label{sec:classification}
 For $0 \leq j<i \leq n$ and $h\in G$, define
\[ \barW^h_{ji}:=W(y_1,\cdots,y_{j},x_{j+1},\cdots,x_i,hx_{i+1},\cdots,hx_n)\in S,\]
\[ \tildeW^h_{ji}:=W(x_1^h,\cdots,x_{j}^h,x_{j+1},\cdots,x_{i},hx_{i+1},\cdots,hx_n)\in R.\]
(Recall that $x_i^h=x_i$ if $hx_i=x_i$ and $x_i^h=0$ otherwise.) We also define
\[ \barW^h_{ii}:=W(y_1,\cdots,y_{i},hx_{i+1},\cdots,hx_n),\]
\[ \tildeW^h_{ii}:=W(x_1^h,\cdots,x_i^h,hx_{i+1},\cdots,hx_n).\]
For $i,j\in \{1,\cdots,n\}$, define
\begin{align*}
g^h_{ji}&:=\begin{cases}
\frac{(\barW^h_{j,i}-\barW^h_{j-1,i})-(\barW^h_{j,i-1}-\barW^h_{j-1,i-1})}{(y_j-x_j) (x_i-hx_i)} &{\rm if\;} i\in I_h, \\
0 & {\rm otherwise}
\end{cases} \\
 f^h_{ji}&:=
 \begin{cases}\frac{(\tildeW^h_{j-1,i}-\tildeW^h_{j,i})-(\tildeW^h_{j-1,i-1}-\tildeW^h_{j,i-1})}{(x_j-hx_j)(x_i-hx_i)}& {\rm if\;} i,j\in I_h,\\
 0 & {\rm otherwise}
 \end{cases}\\
 g^h_{ii}&:=
 \begin{cases}\frac{1}{y_i-x_i}\cdot\Big(\frac{\barW^h_{i,i}-\barW^h_{i-1,i-1}}{y_i-hx_i}-\frac{\barW^h_{i-1,i}-\barW^h_{i-1,i-1}}{x_i-hx_i}\Big) & {\rm if\;} i\in I_h, \\
 0 & {\rm otherwise}
 \end{cases}
 \end{align*}
 
\begin{remark}
$g^h_{ji}=0$ if $j>i$, and $f^h_{ji}=0$ if $j \geq i$.
\end{remark}

For $S=k[x_1,\cdots,x_n,y_1,\cdots,y_n]$, define an $S$-linear map
\begin{align}
 \eta_h: S[\theta_1,\cdots,\theta_n,\partial_1,\cdots,\partial_n] & \to S[\theta_1,\cdots,\theta_n,\partial_1,\cdots,\partial_n],\nonumber \\
\theta_I \partial_J& \mapsto  \sum (-1)^{|I|} g^h_{ji} \frac{\partial\theta_I}{\partial\theta_i}\partial_j\partial_J+\sum f^h_{ji}\frac{\partial^2 \theta_I}{\partial\theta_j\partial\theta_i}\partial_J. \label{etasign}
\end{align}
We also let
\begin{equation}\label{eq:exp} 
\exp(\eta_h):=1+\eta_h+\frac{\eta_h^2}{2!}+\cdots\;\;:S[\theta_1,\cdots,\theta_n,\partial_1,\cdots,\partial_n]  \to S[\theta_1,\cdots,\theta_n,\partial_1,\cdots,\partial_n].
\end{equation}
\eqref{eq:exp} is a finite sum due to graded commutativity of $\theta$ and $\partial$, hence well-defined. 

The following classification theorem is a cornerstone of this paper.
\begin{theorem}\label{thm:cocycle}	
$\exp(\eta_h)(\theta_{I_h})$ is a closed element in $hom_{MF(W(y)-W(x))}(\Delta_1,\Delta_h)$.
\end{theorem}

\begin{proof}
Decompose $D=\Dup+\Ddown$, where
\begin{align*} \Dup(f\theta_I\partial_J):=&\sum (y_i-hx_i)\theta_i \cdot f\theta_I\partial_J+(-1)^{|I|+|J|+1}\sum f\theta_I\partial_J \cdot (y_i-x_i)\theta_i,\\
 \Ddown(f\theta_I\partial_J):=&\sum \nabla_i W(hx,y)\partial_i\cdot f\theta_I\partial_J + (-1)^{|I|+|J|+1} \sum f\theta_I\partial_J \cdot \nabla_i W(x,y) \partial_i.
 \end{align*}
Our assertion holds if we show the following:
\[ \Ddown\Big(\frac{\eta_h^k}{k!}\theta_{I_h}\Big)=-\Dup\Big(\frac{\eta_h^{k+1}}{(k+1)!}\theta_{I_h}\Big).\]
Throughout the proof, we let $g_{ji}:=g^h_{ji}$ and $f_{ji}:=f^h_{ji}$ for simplicity. First we verify
\begin{equation}\label{eq:kthpower}
 \frac{\eta_h^k}{k!}\theta_{I_h}=\sum_{\stackrel{l+m=k}{1\leq i_\bullet,j_\bullet\leq n}} (-1)^{\epsilon_l}g_{j_1 i_1}\cdots g_{j_l i_l} f_{i_{l+1} i_{l+2}} \cdots f_{i_{l+2m-1} i_{l+2m}}\cdot \frac{\partial^{l+2m}\theta_{I_h}}{\partial\theta_{i_1}\cdots\partial\theta_{i_{l+2m}}} \partial_{j_1}\cdots \partial_{j_l},
\end{equation}
with the sign given by 
\begin{equation}\epsilon_l:=l|I_h|+\frac{l(l-1)}{2}. \label{epsilonsign}
\end{equation}
If $l+m=k$, the coefficient $g_{j_1 i_1}\cdots g_{j_l i_l} f_{i_{l+1} i_{l+2}} \cdots f_{i_{l+2m-1} i_{l+2m}}$ occurs $k!$-times as we apply $\eta_h^k$ on $\theta_{I_h}$. To compare signs, observe that
\[ g_{lk}g_{ji} \frac{\partial^2 \theta_{I_h}}{\partial\theta_k\partial\theta_i}\partial_l\partial_j=g_{ji}g_{lk} \frac{\partial^2 \theta_{I_h}}{\partial\theta_i\partial\theta_k}\partial_j\partial_l,\]
\[ f_{kl}g_{ji} \frac{\partial^3 \theta_{I_h}}{\partial\theta_k\partial\theta_l\partial\theta_i}\partial_j=g_{ji}f_{kl} \frac{\partial^3 \theta_{I_h}}{\partial\theta_i\partial\theta_k\partial\theta_l}\partial_j,\]
\[ f_{kl}f_{ij}\frac{\partial^4 \theta_{I_h}}{\partial\theta_k\partial\theta_l\partial\theta_i\partial\theta_j}= f_{ij}f_{kl}\frac{\partial^4 \theta_{I_h}}{\partial\theta_i\partial\theta_j\partial\theta_k\partial\theta_l}\]
by graded commutativity \eqref{eq:cliffrel}. By induction, we conclude that the terms of ${\eta_h^k}\theta_{I_h}$ with coefficients $g_{j_1 i_1}\cdots g_{j_l i_l} f_{i_{l+1} i_{l+2}} \cdots f_{i_{l+2m-1} i_{l+2m}}$ are all the same including their signs. Dividing by $k!$, we get \eqref{eq:kthpower}.

We have
\begin{align}
& \Ddown\Big( \frac{\eta_h^k}{k!}\theta_{I_h}\Big) \label{eq:ddown} \\
=& \sum_{\stackrel{1\leq i \leq n}{l+m=k,1\leq i_\bullet,j_\bullet\leq n}}\nabla_i W(hx,y)\partial_i\cdot (-1)^{\epsilon_l}g_{j_1 i_1}\cdots g_{j_l i_l} f_{i_{l+1} i_{l+2}} \cdots f_{i_{l+2m-1} i_{l+2m}}\cdot \frac{\partial^{l+2m}\theta_{I_h}}{\partial\theta_{i_1}\cdots\partial\theta_{i_{l+2m}}} \partial_{j_1}\cdots \partial_{j_l} \nonumber \\
&+  \sum_{\stackrel{1\leq j \leq n}{l+m=k,1\leq i_\bullet,j_\bullet\leq n}}(-1)^{|I_h|+1+\epsilon_l}g_{j_1 i_1}\cdots g_{j_l i_l} f_{i_{l+1} i_{l+2}} \cdots f_{i_{l+2m-1} i_{l+2m}}\cdot \frac{\partial^{l+2m}\theta_{I_h}}{\partial\theta_{i_1}\cdots\partial\theta_{i_{l+2m}}} \partial_{j_1}\cdots \partial_{j_l}\cdot(\nabla_j W) \partial_j \nonumber\\
=& \sum (-1)^{|I_h|+\epsilon_l} (\nabla_j W(hx,y)-\nabla_j W)g_{j_1 i_1}\cdots g_{j_l i_l} f_{i_{l+1} i_{l+2}} \cdots f_{i_{l+2m-1} i_{l+2m}}\cdot \frac{\partial^{l+2m}\theta_{I_h}}{\partial\theta_{i_1}\cdots\partial\theta_{i_{l+2m}}} \partial_{j_1}\cdots \partial_{j_l}\partial_j \nonumber \\
&+\sum(-1)^{\epsilon_l} \nabla_i W(hx,y)\cdot g_{j_1 i_1}\cdots g_{j_l i_l} f_{i_{l+1} i_{l+2}} \cdots f_{i_{l+2m-1} i_{l+2m}}\cdot \frac{\partial^{l+2m+1}\theta_{I_h}}{\partial\theta_i\partial\theta_{i_1}\cdots\partial\theta_{i_{l+2m}}} \partial_{j_1}\cdots \partial_{j_l}. \nonumber
\end{align}

Recall that we need to show
\[ \eqref{eq:ddown}=-\Dup\Big(\frac{\eta_h^{k+1}}{(k+1)!}\theta_{I_h}\Big).\]
We consider a sum over a subcollection of summands of $\frac{\eta_h^{k+1}}{(k+1)!}\theta_{I_h}$ as follows.
\begin{align*}
&\Psi_{(j_1,i_1),\cdots,(j_l,i_l)}^{(i_{l+1},i_{l+2}),\cdots,(i_{l+2m-1},i_{l+2m})}\\
:=&\sum_{i,j}\big((-1)^{\epsilon_{l+1}}g_{ji}g_{j_1 i_1}\cdots g_{j_l i_l} f_{i_{l+1} i_{l+2}} \cdots f_{i_{l+2m-1} i_{l+2m}}\cdot \frac{\partial^{l+2m+1}\theta_{I_h}}{\partial\theta_i\partial\theta_{i_1}\cdots\partial\theta_{i_{l+2m}}} \partial_j\partial_{j_1}\cdots \partial_{j_l} \\
&\quad +(-1)^{\epsilon_{l}}f_{ji}g_{j_1 i_1}\cdots g_{j_l i_l} f_{i_{l+1} i_{l+2}} \cdots f_{i_{l+2m-1} i_{l+2m}}\cdot \frac{\partial^{l+2m+2}\theta_{I_h}}{\partial\theta_j\partial\theta_i\partial\theta_{i_1}\cdots\partial\theta_{i_{l+2m}}} \partial_{j_1}\cdots \partial_{j_l}\big).
\end{align*}
Then
\[ \frac{\eta_h^{k+1}}{(k+1)!}\theta_{I_h}= \sum_{\stackrel{l+m=k}{1\leq i_\bullet,j_\bullet\leq n}}
\Psi_{(j_1,i_1),\cdots,(j_l,i_l)}^{(i_{l+1},i_{l+2}),\cdots,(i_{l+2m-1},i_{l+2m})}.\] 
\begin{notation}
For $\Psi=\sum\limits_{I,J} f_{IJ} \theta_I \partial_J \in S[\theta_1,\cdots,\theta_n,\partial_1,\cdots,\partial_n]$, denote the summand $f_{IJ}\theta_I \partial_J$ by $\langle \Psi,\theta_I\partial_J\rangle$. Be careful that $\langle \Psi,\theta_I\partial_J\rangle$ is not just $f_{IJ}$.
\end{notation}


We observe that
\begin{align}
& \langle \Dup(\Psi_{(j_1,i_1),\cdots,(j_l,i_l)}^{(i_{l+1},i_{l+2}),\cdots,(i_{l+2m-1},i_{l+2m})}),\frac{\partial^{l+2m}\theta_{I_h}}{\partial\theta_{i_1}\cdots\partial\theta_{i_{l+2m}}} \partial_j\partial_{j_1}\cdots \partial_{j_l} \rangle  \label{eq:dup1}\\
=& \sum_i (x_i-hx_i)\cdot (-1)^{\epsilon_{l+1}}g_{ji}g_{j_1 i_1}\cdots g_{j_l i_l} f_{i_{l+1} i_{l+2}} \cdots f_{i_{l+2m-1} i_{l+2m}}\cdot \theta_i\frac{\partial^{l+2m+1}\theta_{I_h}}{\partial\theta_i\partial\theta_{i_1}\cdots\partial\theta_{i_{l+2m}}} 
\partial_j\partial_{j_1}\cdots \partial_{j_l}\nonumber\\
=& \sum_{ \stackrel{1\leq i \leq n}{i\neq i_1,\cdots,i_{l+2m}}} (x_i-hx_i)\cdot (-1)^{\epsilon_{l+1}} g_{ji}g_{j_1 i_1}\cdots g_{j_l i_l} f_{i_{l+1} i_{l+2}} \cdots f_{i_{l+2m-1} i_{l+2m}}\cdot \frac{\partial^{l+2m}\theta_{I_h}}{\partial\theta_{i_1}\cdots\partial\theta_{i_{l+2m}}} 
\partial_j\partial_{j_1}\cdots \partial_{j_l}. \nonumber
\end{align}
In the last expression, there is no contribution for $i=i_1,\cdots,i_{l+2m}$, because in each of such cases $\theta_{I_h}$ is differentiated twice by same variable $\theta_i$. Nevertheless, we add such dummy terms to \eqref{eq:dup1} and consider
\begin{equation}\label{eq:dup1irr}
\sum_{1\leq i \leq n} (x_i-hx_i) \cdot (-1)^{\epsilon_{l+1}}g_{ji}g_{j_1 i_1}\cdots g_{j_l i_l} f_{i_{l+1} i_{l+2}} \cdots f_{i_{l+2m-1} i_{l+2m}}\cdot \frac{\partial^{l+2m}\theta_{I_h}}{\partial\theta_{i_1}\cdots\partial\theta_{i_{l+2m}}} 
\partial_j\partial_{j_1}\cdots \partial_{j_l}. 
\end{equation}
We show that the dummy terms added in \eqref{eq:dup1irr} also occur as dummy terms pairwise for other summands of $\Dup\Big(\frac{\eta_h^{k+1}}{(k+1)!}\theta_{I_h}\Big)$ with opposite signs, so that they eventually cancel in $\Dup\Big(\frac{\eta_h^{k+1}}{(k+1)!}\theta_{I_h}\Big).$ 

Without loss of generality, let $i=i_1$ in \eqref{eq:dup1irr}. 
We consider another summation $\Psi_{(j,i_1),(j_2,i_2),\cdots,(j_l,i_l)}^{(i_{l+1},i_{l+2}),\cdots,(i_{l+2m-1},i_{l+2m})}$ as a part of $\frac{\eta_h^{k+1}}{(k+1)!}\theta_{I_h}$. Then
\begin{align}
& \langle \Dup\big(\Psi_{(j,i_1),(j_2,i_2),\cdots,(j_l,i_l)}^{(i_{l+1},i_{l+2}),\cdots,(i_{l+2m-1},i_{l+2m})}\big),\frac{\partial^{l+2m}\theta_{I_h}}{\partial\theta_{i_1}\cdots\partial\theta_{i_{l+2m}}} \partial_{j_1}\partial_{j}\partial_{j_2}\cdots \partial_{j_l} \rangle \nonumber\\
=& \sum_{\stackrel{1\leq i \leq n}{i\neq i_1,\cdots,i_{l+2m}}} (1-h)x_i \cdot (-1)^{\epsilon_{l+1}}g_{j_1i}g_{ji_1 }\cdots g_{j_l i_l} f_{i_{l+1} i_{l+2}} \cdots f_{i_{l+2m-1} i_{l+2m}}\cdot \frac{\partial^{l+2m}\theta_{I_h}}{\partial\theta_{i_1}\cdots\partial\theta_{i_{l+2m}}} 
\partial_{j_1}\partial_j\partial_{j_2}\cdots \partial_{j_l}. \label{eq:dup2}
\end{align}
In \eqref{eq:dup2}, the $i_1$th summand 
\begin{equation}\label{eq:i1th} 
(1-h)x_i \cdot (-1)^{\epsilon_{l+1}}g_{j_1i}g_{ji_1 }\cdots g_{j_l i_l} f_{i_{l+1} i_{l+2}} \cdots f_{i_{l+2m-1} i_{l+2m}}\cdot \frac{\partial^{l+2m}\theta_{I_h}}{\partial\theta_{i_1}\cdots\partial\theta_{i_{l+2m}}} 
\partial_{j_1}\partial_j\partial_{j_2}\cdots \partial_{j_l} 
\end{equation}
does not appear, but \eqref{eq:i1th} equals the dummy $i_1$th summand for \eqref{eq:dup1} with the opposite sign. Signs are opposite because the positions of $\partial_j$ and $\partial_{j_1}$ are swapped from each other. So nevertheless we add dummy terms for \eqref{eq:dup1} and \eqref{eq:dup2} for $i=i_1$, they cancel each other in $\Dup\Big(\frac{\eta_h^{k+1}}{(k+1)!}\theta_{I_h}\Big)$. There is no other dummy $i_1$th term which equals to \eqref{eq:i1th} in another partial summation in $\frac{\eta_h^{k+1}}{(k+1)!}\theta_{I_h}$.
 Same argument shows that adding the $i_2$th, $\cdots$, $i_l$th summand do not affect the computation of whole $\Dup\Big(\frac{\eta_h^{k+1}}{(k+1)!}\theta_{I_h}\Big)$.

Now we deal with cases $i=i_k$ for $k=l+1,\cdots,l+2m$. First we let $i=i_{l+2p-1}$ for $p=1,\cdots,m$. Then consider $\Psi_{(j,i_{l+2p-1}),(j_1,i_1),\cdots,(j_l,i_l)}^{(i_{l+1},i_{l+2}),\cdots,\widehat{(i_{l+2p-1},i_{l+2p})},\cdots,(i_{l+2m-1},i_{l+2m})}$, where the hat notation means the exclusion of indices under the hat. Compute
\begin{align}
& \langle \Dup\big(\Psi_{(j,i_{l+2p-1}),(j_1,i_1),\cdots,(j_l,i_l)}^{(i_{l+1},i_{l+2}),\cdots,\widehat{(i_{l+2p-1},i_{l+2p})},\cdots,(i_{l+2m-1},i_{l+2m})}\big),\frac{\partial^{l+2m}\theta_{I_h}}{\partial\theta_{i_{l+2p}}\partial\theta_{i_{l+2p-1}}\partial\theta_{i_1}\cdots\partial\theta_{i_{l+2p-2}}\partial\theta_{i_{l+2p+1}}\cdots\partial\theta_{i_{l+2m}} }\partial_{j}\partial_{j_1}\cdots \partial_{j_l} \rangle\label{eq:dup3}\\
=&\sum_{\stackrel{1\leq i \leq n}{i\neq i_1,\cdots,i_{l+2m}}} (1-h)x_i\cdot (-1)^{\epsilon_{l+1}}(f_{i, i_{l+2p}}-f_{i_{l+2p},i})g_{ji_{l+2p-1}}g_{j_1 i_1}\cdots g_{j_l i_l} f_{i_{l+1} i_{l+2}} \cdots f_{i_{l+2p-3}i_{l+2p-2}}f_{i_{l+2p+1}i_{l+2p+2}}\cdots f_{i_{l+2m-1} i_{l+2m}}\nonumber\\
&\;\;\;\;\cdot \frac{\partial^{l+2m}\theta_{I_h}}{\partial\theta_{i_{l+2p}}\partial\theta_{i_{l+2p-1}}\partial\theta_{i_1}\cdots\partial\theta_{i_{l+2p-2}}\partial\theta_{i_{l+2p+1}}\cdots\partial\theta_{i_{l+2m}} }\partial_{j}\partial_{j_1}\cdots \partial_{j_l}. \nonumber
\end{align}
In \eqref{eq:dup3}, there is no contribution by $i=i_{l+2p-1}$. If we add such a dummy term, then it cancels with the $i_{l+2p-1}$th dummy term for \eqref{eq:dup1}. Their signs are opposite, because the positions of $\partial\theta_{i_{l+2p-1}}$ and $\partial\theta_{i_{l+2p}}$ are swapped. We can repeat the same argument for $i=i_{l+2p}$ for $p=1,\cdots,m$. 

We conclude that
\begin{align}
& \langle\Dup\Big(\frac{\eta_h^{k+1}}{(k+1)!}\theta_{I_h}\Big),\frac{\partial^{l+2m}\theta_{I_h}}{\partial\theta_{i_1}\cdots\partial\theta_{i_{l+2m}}} \partial_j\partial_{j_1}\cdots \partial_{j_l} \rangle \label{eq:dupfirst}\\
=&\sum_{\stackrel{1\leq i \leq n}{l+m=k,1\leq i_\bullet,j_\bullet\leq n}} (1-h)x_i\cdot(-1)^{\epsilon_{l+1}} \big(g_{ji}g_{j_1 i_1}\cdots g_{j_l i_l} f_{i_{l+1} i_{l+2}} \cdots f_{i_{l+2m-1} k_{l+2m}}\nonumber \\
&\;\;\;\;\;\;\;\;\;\;\;\;\;\;\;\;\;\;\;\;\;\;\;\;\;\;\;\;\;+ f_{i,i_{l+1}}g_{ji_{l+2}j}g_{j_1 i_1}\cdots g_{j_l i_l} f_{i_{l+3} i_{l+4}}\cdots f_{i_{l+2m-1}i_{l+2m}}\nonumber \\
& \;\;\;\;\;\;\;\;\;\;\;\;\;\;\;\;\;\;\;\;\;\;\;\;\;\;\;\;\;- f_{i_{l+1},i}g_{ji_{l+2}}g_{j_1 i_1}\cdots g_{j_l i_l} f_{i_{l+3} i_{l+4}}\cdots f_{i_{l+2m-1}i_{l+2m}}\big)\cdot \frac{\partial^{l+2m}\theta_{I_h}}{\partial\theta_{i_1}\cdots\partial\theta_{i_{l+2m}}}\partial_j\partial_{j_1}\cdots \partial_{j_l}. \nonumber
\end{align}

Next, we compute
\begin{align}
& \langle \Dup(\Psi_{(j_1,i_1),\cdots,(j_l,i_l)}^{(i_{l+1},\cdots,i_{l+2m})}),
\frac{\partial^{l+2m}\theta_{I_h}}{\partial\theta_i\partial\theta_{i_1}\cdots\partial\theta_{i_{l+2m}}} \partial_{j_1}\cdots \partial_{j_l} \rangle 
\label{eq:dup4} \\
=& \sum_{j\neq j_1,\cdots,j_l} (y_j-x_j) g_{ji}\cdot (-1)^{\epsilon_{l+1}+l+|I_h|+1}g_{j_1 i_1}\cdots g_{j_l i_l} f_{i_{l+1} i_{l+2}} \cdots f_{i_{l+2m-1} i_{l+2m}}\cdot \frac{\partial^{l+2m+1}\theta_{I_h}}{\partial\theta_i\partial\theta_{i_1}\cdots\partial\theta_{i_{l+2m}}} 
\Big(\frac{\partial \theta_j}{\partial \theta_j}\Big)\partial_{j_1}\cdots \partial_{j_l}  \nonumber\\
&+ \sum_{j\neq i_1,\cdots,i_{l+2m}} (1-h)x_j f_{ij}\cdot (-1)^{\epsilon_l} g_{j_1 i_1}\cdots g_{j_l i_l} f_{i_{l+1} i_{l+2}} \cdots f_{i_{l+2m-1} i_{l+2m}}\cdot \theta_j \frac{\partial^{l+2m+2}\theta_{I_h}}{\partial\theta_i\partial\theta_j\partial\theta_{i_1}\cdots\partial\theta_{i_{l+2m}}} 
\partial_{j_1}\cdots \partial_{j_l}\nonumber\\ 
&+ \sum_{k\neq i_1,\cdots,i_{l+2m}} (1-h)x_k f_{ki}\cdot (-1)^{\epsilon_l} g_{j_1 i_1}\cdots g_{j_l i_l} f_{i_{l+1} i_{l+2}} \cdots f_{i_{l+2m-1} i_{l+2m}}\cdot \theta_k\frac{\partial^{l+2m+2}\theta_{I_h}}{\partial\theta_k\partial\theta_i\partial\theta_{i_1}\cdots\partial\theta_{i_{l+2m}}} 
\partial_{j_1}\cdots \partial_{j_l}.\nonumber \\
=& \sum_{j\neq j_1,\cdots,j_l} (y_j-x_j) g_{ji}\cdot (-1)^{\epsilon_{l+1}+l+|I_h|+1}g_{j_1 i_1}\cdots g_{j_l i_l} f_{i_{l+1} i_{l+2}} \cdots f_{i_{l+2m-1} i_{l+2m}}\cdot \frac{\partial^{l+2m+1}\theta_{I_h}}{\partial\theta_i\partial\theta_{i_1}\cdots\partial\theta_{i_{l+2m}}} 
\partial_{j_1}\cdots \partial_{j_l}  \nonumber\\
&+ \sum_{j\neq i_1,\cdots,i_{l+2m}} (1-h)x_j f_{ij}\cdot (-1)^{\epsilon_l+1} g_{j_1 i_1}\cdots g_{j_l i_l} f_{i_{l+1} i_{l+2}} \cdots f_{i_{l+2m-1} i_{l+2m}}\cdot \frac{\partial^{l+2m+1}\theta_{I_h}}{\partial\theta_i\partial\theta_{i_1}\cdots\partial\theta_{i_{l+2m}}} 
\partial_{j_1}\cdots \partial_{j_l}\nonumber\\ 
&+ \sum_{k\neq i_1,\cdots,i_{l+2m}} (1-h)x_k f_{ki}\cdot (-1)^{\epsilon_l} g_{j_1 i_1}\cdots g_{j_l i_l} f_{i_{l+1} i_{l+2}} \cdots f_{i_{l+2m-1} i_{l+2m}}\cdot \frac{\partial^{l+2m+1}\theta_{I_h}}{\partial\theta_i\partial\theta_{i_1}\cdots\partial\theta_{i_{l+2m}}} 
\partial_{j_1}\cdots \partial_{j_l}.\nonumber
\end{align}
Again, in each summation of \eqref{eq:dup4} some indices do not contribute. As above, we will consider such dummy terms altogether and show that they eventually cancel each other. In the first summation of \eqref{eq:dup4}, consider $j=j_1$ without loss of generality. For another summand $\Psi_{(j_1,i),\cdots,(j_l,i_l)}^{(i_{l+1},\cdots,i_{l+2m})}$, we have
\begin{align}
& \langle \Dup(\Psi_{(j_1,i),\cdots,(j_l,i_l)}^{(i_{l+1},\cdots,i_{l+2m})}),
\frac{\partial^{l+2m}\theta_{I_h}}{\partial\theta_i\partial\theta_{i_1}\cdots\partial\theta_{i_{l+2m}}} \partial_{j_1}\cdots \partial_{j_l} \rangle \label{eq:dup5}\\
=& \sum_{j\neq j_1,\cdots,j_l} (y_j-x_j) g_{ji_1 }\cdot (-1)^{\epsilon_{l+1}+l+|I_h|+1}g_{j_1 i}\cdots g_{j_l i_l} f_{i_{l+1} i_{l+2}} \cdots f_{i_{l+2m-1} i_{l+2m}}\cdot \frac{\partial^{l+2m+1}\theta_{I_h}}{\partial\theta_{i_1}\partial\theta_{i}\cdots\partial\theta_{i_{l+2m}}} 
\partial_{j_1}\cdots \partial_{j_l}  \nonumber\\
&+ \sum_{j\neq i, i_1,\cdots,i_{l+2m}} (1-h)x_j f_{i_1 j}\cdot (-1)^{\epsilon_l+1} g_{j_1 i}\cdots g_{j_l i_l} f_{i_{l+1} i_{l+2}} \cdots f_{i_{l+2m-1} i_{l+2m}}\cdot \frac{\partial^{l+2m+1}\theta_{I_h}}{\partial\theta_{i_1}\partial\theta_{i}\cdots\partial\theta_{i_{l+2m}}} 
\partial_{j_1}\cdots \partial_{j_l}\nonumber\\ 
&+ \sum_{k\neq i, i_1,\cdots,i_{l+2m}} (1-h)x_k f_{ki_1}\cdot (-1)^{\epsilon_l}g_{j_1 i}\cdots g_{j_l i_l} f_{i_{l+1} i_{l+2}} \cdots f_{i_{l+2m-1} i_{l+2m}}\cdot \frac{\partial^{l+2m+1}\theta_{I_h}}{\partial\theta_{i_1}\partial\theta_{i}\cdots\partial\theta_{i_{l+2m}}} 
\partial_{j_1}\cdots \partial_{j_l}.\nonumber
\end{align}
The $j_1$th dummy term for \eqref{eq:dup5} coincides with the $j_1$th dummy term for \eqref{eq:dup4} with the opposite sign. 

Likewise, in the second summation of \eqref{eq:dup4} the $i_1$th summand does not contribute. Consider
\begin{align}
& \langle \Dup\big(\Psi_{(j_2,i_2),\cdots,(j_l,i_l)}^{(i,i_{1}),(i_{l+1},i_{l+2}),\cdots,(i_{l+2m-1},i_{l+2m})}\big),\frac{\partial^{l+2m+1}\theta_{I_h}}{\partial\theta_i\partial\theta_{i_1}\cdots\partial\theta_{i_{l+2m}}}\partial_{j_1}\cdots \partial_{j_l} \rangle \label{eq:dup6} \\
=& \sum_{k\neq i,i_1,\cdots,i_{l+2m}} (1-h)x_k\cdot (-1)^{\epsilon_{l}}g_{j_1k}g_{j_2 i_2}\cdots g_{j_l i_l} f_{ii_1}f_{i_{l+1} i_{l+2}} \cdots f_{i_{l+2m-1} i_{l+2m}}\frac{\partial^{l+2m+1}\theta_{I_h}}{\partial\theta_{i_2}\cdots \partial\theta_{i_l}\partial\theta_i\partial\theta_{i_1} \partial\theta_{i_{l+1}}\cdots\partial\theta_{i_{l+2m}}} 
\partial_{j_1}\cdots \partial_{j_l}\nonumber \\
=& \sum_{k\neq i,i_1,\cdots,i_{l+2m}} (1-h)x_k\cdot (-1)^{\epsilon_{l}}g_{j_1k}g_{j_2 i_2}\cdots g_{j_l i_l} f_{ii_1}f_{i_{l+1} i_{l+2}} \cdots f_{i_{l+2m-1} i_{l+2m}}\frac{\partial^{l+2m+1}\theta_{I_h}}{\partial\theta_{i} \partial\theta_{i_1}\cdots\partial\theta_{i_{l+2m}}} 
\partial_{j_1}\cdots \partial_{j_l}.\nonumber
\end{align}
Again, the dummy $i_1$th summand for \eqref{eq:dup6} coincides with the dummy $i_1$th summand for \eqref{eq:dup5} with opposite sign. We can apply the same argument for the second summation of \eqref{eq:dup4}. Therefore, we modify \eqref{eq:dup4} adding all dummy terms and get the same result for $ \Dup\Big(\frac{\eta_h^{k+1}}{(k+1)!}\theta_{I_h}\Big)$.

By definition \eqref{epsilonsign}, we have 
\[ \epsilon_{l+1}+l+|I_h|+1=\epsilon_l+1,\]
and adding all dummy terms, we can rewrite \eqref{eq:dup4} as
\begin{align*}
&\langle \Dup\Big(\frac{\eta_h^{k+1}}{(k+1)!}\theta_{I_h}\Big),
\frac{\partial^{l+2m+1}\theta_{I_h}}{\partial\theta_i\partial\theta_{i_1}\cdots\partial\theta_{i_{l+2m}}} 
\partial_{j_1}\cdots \partial_{j_l} \rangle \\
=&\;\;  \big(\sum_{j=1}^n (y_j-x_j) g_{ji}+ \sum_{j:j>i} (1-h)x_j f_{ij}-\sum_{k:k<i} (1-h)x_j f_{ki}\big)\\
&\;\;\cdot (-1)^{\epsilon_l+1} g_{j_1 i_1}\cdots g_{j_l i_l} f_{i_{l+1} i_{l+2}} \cdots f_{i_{l+2m-1} i_{l+2m}}\cdot\frac{\partial^{l+2m+1}\theta_{I_h}}{\partial\theta_i\partial\theta_{i_1}\cdots\partial\theta_{i_{l+2m}}} 
\partial_{j_1}\cdots \partial_{j_l}.
\end{align*}

Finally we are at the stage of wrapping up the proof.
By definition of $g_{ji}$, we have
\[ \sum_{i= j}^n (x_i-hx_i) g_{ji}= \nabla_j W- \nabla_j W(hx,y).\]
Also,
\[ \sum_{j =1}^i(y_j-x_j) g_{ji}=  \nabla_i W(hx,y)-\frac{\barW_{0,i}-\barW_{0,i-1}}{x_i-hx_i}.\]
Note that 
\begin{equation}\label{eq:bartilde}
\barW_{0,i}=\tildeW_{0,i}.
\end{equation}
 Also by $G$-invariance of $W$, for any $i$,
\begin{align}
 \tildeW_{i,i}&=W(x_1^h,\cdots,x_i^h,hx_{i+1},\cdots,hx_n)\nonumber\\
 &=W(hx_1^h,\cdots,hx_i^h,hx_{i+1},\cdots,hx_n)\nonumber\\
 &=W(x_1^h,\cdots,x_i^h,x_{i+1},\cdots,x_n)\nonumber\\
 &=\tildeW_{i,n}. \label{eq:Winv}
 \end{align}
Consequently, we have
\begin{align}
& \sum_{k=1}^{i-1} (x_k-hx_k) f_{ki} -\sum_{j=i+1}^n (x_j-hx_j) f_{ij}\nonumber  \\
=&\sum_{k=1}^{i-1}\frac{(\tildeW_{k-1,i}-\tildeW_{k,i})-(\tildeW_{k-1,i-1}-\tildeW_{k,i-1})}{x_i-hx_i}
- \sum_{j=i+1}^n\frac{(\tildeW_{i-1,j}-\tildeW_{i-1,j-1})-(\tildeW_{i,j}-\tildeW_{i,j-1})}{x_i-hx_i}\nonumber\\
=&\frac{(\tildeW_{0,i}-\tildeW_{i-1,i})-(\tildeW_{0,i-1}-\tildeW_{i-1,i-1})-(\tildeW_{i-1,n}-\tildeW_{i-1,i})+(\tildeW_{i,n}-\tildeW_{i,i})}{x_i-hx_i} \nonumber \\
=& \frac{\tildeW_{0,i}-\tildeW_{0,i-1}+\tildeW_{i-1,i-1}-\tildeW_{i-1,n}+\tildeW_{i,n}-\tildeW_{i,i}}{x_i-hx_i} \label{eq:fup}
\end{align}
Above observations \eqref{eq:bartilde} and \eqref{eq:Winv} show that
\[ \eqref{eq:fup}=\frac{\barW_{0,i}-\barW_{0,i-1}}{x_i-hx_i}.\]
By \eqref{eq:ddown}, 
\begin{align*}
&\langle \Ddown\Big( \frac{\eta_h^k}{k!}\theta_{I_h}\Big),\frac{\partial^{l+2m}\theta_{I_h}}{\partial\theta_{i_1}\cdots\partial\theta_{i_{l+2m}}} \partial_j \partial_{j_1}\cdots \partial_{j_l}\rangle \\
=& \sum (-1)^{|I_h|+\epsilon_l+l} (\nabla_j W(hx,y)-\nabla_j W)g_{j_1 i_1}\cdots g_{j_l i_l} f_{i_{l+1} i_{l+2}} \cdots f_{i_{l+2m-1} i_{l+2m}}\cdot \frac{\partial^{l+2m}\theta_{I_h}}{\partial\theta_{i_1}\cdots\partial\theta_{i_{l+2m}}} \partial_j\partial_{j_1}\cdots \partial_{j_l}.
\end{align*}
We fix indices $(i_r,j_r)$ and $(i_s,i_{s+1})$ for $r=1,\cdots,l$ and $s=l,\cdots,l+2m-1$ of the summation. Fixing same indices in \eqref{eq:dupfirst}, we have 
\begin{align*}
& \langle \Dup \Big(\frac{\eta_h^{k+1}}{(k+1)!}\theta_{I_h}\Big),\frac{\partial^{l+2m}\theta_{I_h}}{\partial\theta_{i_1}\cdots\partial\theta_{i_{l+2m}}} \partial_j \partial_{j_1}\cdots \partial_{j_l}\rangle \\
=& \sum_i (x_i-hx_i)\cdot(-1)^{\epsilon_{l+1}} g_{ji}g_{j_1 i_1}\cdots g_{j_l i_l} f_{i_{l+1} i_{l+2}} \cdots f_{i_{l+2m-1} i_{l+2m}}\frac{\partial^{l+2m}\theta_{I_h}}{\partial\theta_{i_1}\cdots\partial\theta_{i_{l+2m}}}\partial_j\partial_{j_1}\cdots \partial_{j_l} \\
=& (-1)^{\epsilon_{l+1}}\big(\nabla_j W- \nabla_j W(hx,y)\big)g_{j_1 i_1}\cdots g_{j_l i_l} f_{i_{l+1} i_{l+2}} \cdots f_{i_{l+2m-1} i_{l+2m}}\frac{\partial^{l+2m}\theta_{I_h}}{\partial\theta_{i_1}\cdots\partial\theta_{i_{l+2m}}}\partial_j\partial_{j_1}\cdots \partial_{j_l} \\
=& (-1)^{|I_h|+\epsilon_l+l+1}\big(\nabla_j W(hx,y)-\nabla_j W \big)g_{j_1 i_1}\cdots g_{j_l i_l} f_{i_{l+1} i_{l+2}} \cdots f_{i_{l+2m-1} i_{l+2m}}\frac{\partial^{l+2m}\theta_{I_h}}{\partial\theta_{i_1}\cdots\partial\theta_{i_{l+2m}}}\partial_j\partial_{j_1}\cdots \partial_{j_l}\\
=&- \langle \Ddown\Big( \frac{\eta_h^k}{k!}\theta_{I_h}\Big),\frac{\partial^{l+2m}\theta_{I_h}}{\partial\theta_{i_1}\cdots\partial\theta_{i_{l+2m}}} \partial_j \partial_{j_1}\cdots \partial_{j_l}\rangle,
\end{align*}
hence
\[ \Dup \Big(\frac{\eta_h^{k+1}}{(k+1)!}\theta_{I_h}\Big)=-\Ddown\Big( \frac{\eta_h^k}{k!}\theta_{I_h}\Big). \qedhere\]
\end{proof}


\begin{corollary}\label{cor:rank1rep}
Any cohomology class in $\Hom(\Delta_1,\Delta_h)$ is represented by an element $f(x)\cdot \exp(\eta_h)\theta_{I_h}$ for $f\in R$. In other words, $\exp(\eta_h)\theta_{I_h}$ is a generator of $\Hom(\Delta_1,\Delta_h)$ as a rank $1$ free $\Jac(W^h)$-module.
\end{corollary}

\begin{proof}
By Lemma \ref{lemma:homjac} we know that $\Hom(\Delta_1,\Delta_h)$ is a free $\Jac(W^h)$-module of rank 1. $\Jac(W^h)$ is finite dimensional over $k$ because $W$ has the isolated singularity. Therefore a rank 1 free $\Jac(W^h)$-module does not have any proper free submodule: any element of $\Jac(W^h)$ is not a zero-divisor if and only if it is a unit.
Define a module homomorphism
\begin{align} 
\Jac(W^h)\cdot \xi_h & \to \Hom(\Delta_1,\Delta_h),\nonumber\\
f(x)\cdot \xi_h & \mapsto f(x)\cdot \exp(\eta_h)(\theta_{I_h}). \label{eq:jachom}
\end{align}
We verify that \eqref{eq:jachom} is an injective module homomorphism. Suppose that $f(x)\cdot \exp(\eta_h)(\theta_{I_h})$ is nullhomotopic. The morphism $f(x)\cdot \exp(\eta_h)(\theta_{I_h})$ restricts to $f(x)\cdot \theta_{I_h}$ on $S \subset \Delta_1=S[\theta_1,\cdots,\theta_n]$, and a homotopy is given as follows.
\[\displaystyle\xymatrix{
& & S \ar[ddll]_(.35){s_1}\ar[dd]^{f(x)\cdot\theta_{I_h}} \ar[rr]^-{\sum_i (y_i-x_i)\theta_i} \ar[ddrr]^(.35){s_2} & & \bigoplus_i S\cdot \theta_i \ar[ddll]^(.2){s_3}\\
& & & & \\
\bigoplus_{i\in I_h} S\cdot \theta_{I_h-\{i\}} \ar[rr]^-{\sum_i (y_i-hx_i)\theta_i} & & S\cdot \theta_{I_h} & & \bigoplus_i S\cdot \theta_{I_h \cup \{i\}} \ar[ll]_-{\sum_i \nabla_i W(hx,y)\partial_i}}
\]
We conclude that
\begin{equation} \label{eq:fxhomotopy}
f(x)=\sum_i (y_i-hx_i)s'_{1,i}+\sum_i \nabla_i W(hx,y) s'_{2,i}+\sum_i (y_i-x_i)s'_{3,i} 
\end{equation}
for some $s'_{1,i},s'_{2,i},s'_{3,i} \in S$ for $i=1,\cdots,n$. Since the left hand side of \eqref{eq:fxhomotopy} is a polynomial only in $x$ variables, the right hand side does not change after we substitute $y_i=hx_i$ for all $i$. Therefore we can rewrite \eqref{eq:fxhomotopy} as 
\begin{align*}
 f(x)&=\sum_i \nabla_i W(hx,y)|_{y=hx} s'_{2,i}(x,hx)+ \sum_{i\in I_h} (hx_i-x_i) s'_{3,i}(x,hx) \\
 &= \sum_i h_i^{-1}\partial_i W\cdot s'_{2,i}(x,hx)+\sum_{i\in I_h}(hx_i-x_i)  s'_{3,i}(x,hx).
 \end{align*}
The first summation is in the Jacobian ideal of $W$, and the second summation is zero on ${\rm Fix}(h)$. So $f(x)$ is zero in $\Jac(W^h)$. Since \eqref{eq:jachom} is injective, it is also surjective because $\Hom(\Delta_1,\Delta_h)$ is a rank 1 free $\Jac(W^h)$-module and thus has no proper free submodule.
\end{proof}

\section{Comparison of multiplications}\label{sec:comparemulti}
We are ready to examine the algebra structure on $\Hom_{1\times G}(\Delta_W^{G\times G},\Delta_W^{G\times G})$. By Lemma \ref{lem:1gequivendo}, we naturally construct the multiplication $\cup$ on $\bigoplus_{g\in G}\Hom(\Delta_1,\Delta_g)$ as follows. Let $\phi \in hom(\Delta_1,\Delta_g)$ and $\psi \in hom(\Delta_1,\Delta_h)$. Then 
\begin{equation}\label{eq:cupprod}
\phi \cup \psi:= (h_*\phi) \circ \psi \in hom(\Delta_1,\Delta_{gh}) 
\end{equation}
where $h_*\phi: \Delta_h \to \Delta_{gh}$ is defined by
\[ (h_* \phi)(r(x,y)v_h):= r(x,y)\cdot (1\times h^{-1})\cdot (\phi(\rho(h)v_1)). \]
The definition of $\cup$ in \eqref{eq:cupprod}  is justified by \eqref{eq:translation}. Namely, given $\phi: \Delta_1 \to \Delta_g$, the $(1\times G)$-equivariant endomorphism of $\Delta_W^{G\times G}$ which restricts to $\phi$ on $\Delta_1$ is $\bigoplus_{h\in G} h_*\phi$. 
\begin{remark}
If $\phi=f(x,y)\theta_I \partial_J \in hom(\Delta_1,\Delta_g)$, then 
\[ h_*\phi= f(x,h^{-1}y)\cdot \rho(h^{-1})(\theta_I\partial_J) \in hom(\Delta_h,\Delta_{hg}).\]
\end{remark}


\begin{lemma}\label{lem:leadingterm}
Let $\phi \in hom_{MF(W(y)-W(x))}(\Delta_1,\Delta_h)$ be a closed element. If the coefficient of $\theta_{I_h}$ in $\phi$ is zero, then $\phi$ is nullhomotopic.
\end{lemma}

\begin{proof}
By Corollary \ref{cor:rank1rep}, $\phi$ is homotopic to $f\cdot \exp(\eta_h)\theta_{I_h}$ for some $f\in R$. Let $s$ be a homotopy between $\phi$ and $f\cdot \exp(\eta_h)\theta_{I_h}$. Consider the following diagram which expresses a part of morphisms from $\Delta_1$ to $\Delta_h$.
\[\displaystyle\xymatrix{
& & S \ar[ddll]_(.35){s_1}\ar@<0.5ex>[dd]^{f\theta_{I_h}} \ar@<-0.5ex>[dd]_-0\ar[rr]^-{\sum_i (y_i-x_i)\theta_i} \ar[ddrr]^(.35){s_2} & & \bigoplus_i S\cdot \theta_i \ar[ddll]^(.2){s_3}\\
& & & & \\
\bigoplus_{i\in I_h} S\cdot \theta_{I_h-\{i\}} \ar[rr]^-{\sum_i (y_i-hx_i)\theta_i} & & S\cdot \theta_{I_h} & & \bigoplus_i S\cdot \theta_{I_h \cup \{i\}} \ar[ll]_-{\sum_i \nabla_i W(hx,y)\partial_i}}
\]
In this part $S$ is mapped to $0$ via $\phi$, while mapped by $f\theta_{I_h}$ via $f\cdot \exp(\eta_h)\theta_{I_h}$. The homotopy $s$ gives rise to
\[ f\theta_{I_h}= \big(\sum_i (y_i-hx_i)\theta_i\big)\cdot s_1 +\big(\sum_i \nabla_i W(hx,y)\partial_i\big)\cdot s_2 +(-1)^{|s_3|+1} s_3 \cdot \sum_i (y_i-x_i)\theta_i,\]
so
\begin{equation}\label{eq:fxy} f(x)=\sum_i (y_i-hx_i)s'_{1,i}+\sum_i \nabla_i W(hx,y) s'_{2,i}+\sum_i (y_i-x_i)s'_{3,i} \end{equation}
for some $s'_{1,i},s'_{2,i},s'_{3,i} \in S$ for $i=1,\cdots,n$. Now the latter part of the proof of Corollary \ref{cor:rank1rep} implies that $f(x)$ is zero inside $\Jac(W^h)$, so $f\cdot \exp(\eta_h)\theta_{I_h}$ is zero in $\Hom(\Delta_1,\Delta_h)$. Therefore, $\phi$ is nullhomotopic.
\end{proof}
Lemma \ref{lem:leadingterm} implies that a cohomology class in $\Hom(\Delta_1,\Delta_h)$ is completely determined by its coefficient of $\theta_{I_h}$. Hence when we compute $\phi\cup \psi$ for $\phi\in \Hom(\Delta_1,\Delta_{g}) $ and $\psi\in \Hom(\Delta_1,\Delta_{h})$, it suffices to compute the coefficient of $\theta_{I_{gh}}$ in $(h_*\phi) \circ \psi$. 

Before introducing our main theorem, we need two computations concerning Clifford variables.
\begin{lemma}
Let $\partial_{i_s}\otimes \partial_{j_s} \in k[\partial_1,\cdots,\partial_n]^{\otimes 2}$ for $s=1,\cdots,k$. Then
\begin{equation}\label{eq:thetatensorprod}
 (\partial_{i_1}\otimes \partial_{j_1}) \cdots (\partial_{i_k}\otimes \partial_{j_k})
=(-1)^{\frac{k(k-1)}{2}}\partial_{i_1}\cdots\partial_{i_k}\otimes \partial_{j_1}\cdots\partial_{j_k}.
\end{equation}
\end{lemma}

\begin{proof}
Note that $(\partial_I\otimes \partial_J)\cdot (\partial_K \otimes \partial_L)=(-1)^{|J||K|}\partial_I\partial_K \otimes \partial_J\partial_L$ in $k[\partial_1,\cdots,\partial_n]^{\otimes 2}$. The proof is finished by induction.
\end{proof}

\begin{lemma}\label{lem:mf-jac-comparison}
For $h_1,h_2\in G$, suppose that we have maps
\[\omega_{h_i}^{(1,1)},\; \omega_{h_i}^{(2,0)}: S[\theta_1,\cdots,\theta_n,\partial_1,\cdots,\partial_n] \to S[\theta_1,\cdots,\theta_n,\partial_1,\cdots,\partial_n]\]
as follows.
\[ \omega_{h_1}^{(1,1)}(\theta_I\partial_J)=\sum_{\stackrel{i\leq j}{i,j\in I_{h_1}}} (-1)^{|I|} a_{ij}\frac{\partial \theta_I}{\partial \theta_j}\partial_i\partial_J,\quad 
\omega_{h_1}^{(2,0)}(\theta_I\partial_J)=\sum_{\stackrel{i<j}{i,j\in I_{h_1}}}b_{ij}\frac{\partial\theta_I}{\partial\theta_i\partial\theta_j}\partial_J,\]
\[ \omega_{h_2}^{(1,1)}(\theta_I\partial_J)=\sum_{\stackrel{i\leq j}{i,j\in I_{h_2}}} (-1)^{|I|} c_{ij}\frac{\partial \theta_I}{\partial \theta_j}\partial_i\partial_J,\quad 
\omega_{h_2}^{(2,0)}(\theta_I\partial_J)=\sum_{\stackrel{i<j}{i,j\in I_{h_2}}}d_{ij}\frac{\partial\theta_I}{\partial\theta_i\partial\theta_j}\partial_J\]
where $a_{ij},b_{ij},c_{ij},d_{ij} \in S$.
Let 
\[\omega_{h_1}:=\omega_{h_1}^{(1,1)}+\omega_{h_1}^{(2,0)}, \quad 
\omega_{h_2}:=\omega_{h_2}^{(1,1)}+\omega_{h_2}^{(2,0)}.\]
Then
\begin{align}
&  \Big\langle\sum_{k+l=n} \frac{\omega_{h_1}^k(\theta_{I_{h_1}})}{k!} \cdot\frac{\omega_{h_2}^l(\theta_{I_{h_2}})}{l!}, \theta_{I_{h_1h_2}} \Big\rangle \label{eq:matrixmulti}\\
=&\Big\langle \frac{1}{n!} \Upsilon \Big( \big( \sum_{\stackrel{i\leq j}{i,j\in I_{h_1}}} (a_{ij}\partial_j\otimes \partial_i+b_{ij} \partial_i\partial_j \otimes 1)+\sum_{\stackrel{i<j}{i,j\in I_{h_2}}}d_{ij} \cdot 1\otimes \partial_i\partial_j\big)^n\otimes \theta_{I_{h_1}}\otimes \theta_{I_{h_2}}\Big), \theta_{I_{h_1h_2}} \Big\rangle. \label{eq:shkformula}
\end{align}

\end{lemma}

\begin{proof}
\begin{align}
&\frac{1}{n!}\big( \sum_{\stackrel{i\leq j}{i,j\in I_{h_1}}} (a_{ij}\partial_j\otimes \partial_i+b_{ij} \partial_i\partial_j \otimes 1)+\sum_{\stackrel{i<j}{i,j\in I_{h_2}}}d_{ij} \cdot 1\otimes \partial_i\partial_j\big)^n\nonumber\\
=&\sum_{k+l=n}\;\;\sum_{k_1+k_2=k}\frac{1}{k_1!}{\sum\limits_{\stackrel{i_\bullet\leq j_\bullet}{i_\bullet,j_\bullet\in I_{h_1}}}(-1)^{\frac{k_1(k_1-1)}{2}}a_{i_1j_1}\cdots a_{i_{k_1}j_{k_1}}\partial_{j_1}\cdots\partial_{j_{k_1}}\otimes \partial_{i_1}\cdots\partial_{i_{k_1}}} \nonumber
\\
&\quad\quad\quad\quad\quad\quad
\cdot\frac{1}{k_2!}{\sum\limits_{\stackrel{i'_\bullet\leq j'_\bullet}{i'_\bullet,j'_\bullet\in I_{h_1}}}b_{i'_1j'_1}\cdots b_{i'_{k_2}j'_{k_2}}\partial_{i'_1}\partial_{j'_1}\cdots\partial_{i'_{k_2}}\partial_{j'_{k_2}}\otimes 1}  \nonumber
\\
&\quad\quad\quad\quad\quad\quad
\cdot\frac{1}{l!}{\sum\limits_{\stackrel{i''_\bullet\leq j''_\bullet}{i''_\bullet,j''_\bullet\in I_{h_2}}}d_{i''_1j''_1}\cdots d_{i''_{l}j''_{l}}\cdot 1\otimes\partial_{i''_1}\partial_{j''_1}\cdots\partial_{i''_{l}}\partial_{j''_{l}}}\nonumber\\
=&\sum_{\stackrel{k+l=n}{k_1+k_2=k}}{\sum\limits_{\stackrel{i_\bullet\leq j_\bullet}{i_\bullet,j_\bullet\in I_{h_1}}}\sum\limits_{\stackrel{i'_\bullet\leq j'_\bullet}{i'_\bullet,j'_\bullet\in I_{h_1}}}\sum\limits_{\stackrel{i''_\bullet\leq j''_\bullet}{i''_\bullet,j''_\bullet\in I_{h_2}}}\frac{(-1)^{\frac{k_1(k_1-1)}{2}}}{k_1!k_2!l!}a_{i_1j_1}\cdots a_{i_{k_1}j_{k_1}}b_{i'_1j'_1}\cdots b_{i'_{k_2}j'_{k_2}}d_{i''_1j''_1}\cdots d_{i''_{l}j''_{l}}} \nonumber
\\
&\quad\quad
\cdot \partial_{j_1}\cdots\partial_{j_{k_1}}\partial_{i'_1}\partial_{j'_1}\cdots\partial_{i'_{k_2}}\partial_{j'_{k_2}}\otimes \partial_{i_1}\cdots\partial_{i_{k_1}}\partial_{i''_1}\partial_{j''_1}\cdots \partial_{i''_l}\partial_{j''_l}. \label{eq:nthpower}
\end{align}
The sign is due to \eqref{eq:thetatensorprod}. Plugging \eqref{eq:nthpower} into \eqref{eq:shkformula}, we get
\begin{align*}
\eqref{eq:shkformula}&= \Big\langle\sum_{\stackrel{k+l=n}{k_1+k_2=k}} {\sum\limits_{\stackrel{i_\bullet\leq j_\bullet}{i_\bullet,j_\bullet\in I_{h_1}}}\sum\limits_{\stackrel{i'_\bullet\leq j'_\bullet}{i'_\bullet,j'_\bullet\in I_{h_1}}}\sum\limits_{\stackrel{i''_\bullet\leq j''_\bullet}{i''_\bullet,j''_\bullet\in I_{h_2}}}\frac{(-1)^{\frac{k_1(k_1-1)}{2}+k_1|I_{h_1}|}}{k_1!k_2!l!}a_{i_1j_1}\cdots a_{i_{k_1}j_{k_1}}b_{i'_1j'_1}\cdots b_{i'_{k_2}j'_{k_2}}d_{i''_1j''_1}\cdots d_{i''_{l}j''_{l}}} \\
&\quad\quad\quad
\cdot \frac{\partial^{k_1+2k_2}\theta_{I_{h_1}}}{\partial\theta_{j_1}\cdots\partial\theta_{j_{k_1}}
\partial\theta_{i'_1}\partial\theta_{j'_1}\cdots\partial\theta_{i'_{k_2}}\partial\theta_{j'_{k_2}}}
\cdot \frac{\partial^{k_1+2l}\theta_{I_{h_2}}}{\partial\theta_{i_1}\cdots\partial\theta_{i_{k_1}}\partial\theta_{i''_1}\partial\theta_{j''_1}\cdots \partial\theta_{i''_l}\partial\theta_{j''_l}}, \theta_{I_{h_1h_2}} \Big\rangle.
\end{align*}

By $\omega_{h_1}^{(1,1)}\omega_{h_1}^{(2,0)}=\omega_{h_1}^{(2,0)}\omega_{h_1}^{(1,1)}$ which is easy to verify, we have
\begin{align}
 \frac{1}{k!}{\omega_{h_1}^k(\theta_{I_{h_1}})} &= 
\sum\limits_{k_1+k_2=k}  \frac{(\omega_{h_1}^{(1,1)})^{k_1}\cdot(\omega_{h_1}^{(2,0)})^{k_2}}{k_1! k_2!}(\theta_{I_{h_1}})\nonumber\\
 &=\sum\limits_{k_1+k_2=k} \frac{1}{k_1!k_2!}\cdot(\omega_{h_1}^{(1,1)})^{k_1}
 \Big( \sum\limits_{\stackrel{i'_\bullet\leq j'_\bullet}{i'_\bullet,j'_\bullet\in I_{h_1}}}b_{i'_1j'_1}\cdots b_{i'_{k_2}j'_{k_2}}\frac{\partial^{2k_2}\theta_{I_{h_1}}}{\partial\theta_{i'_1}\partial\theta_{j'_1}\cdots\partial\theta_{i'_{k_2}}\partial\theta_{j'_{k_2}}} \Big)\nonumber\\
 &=\sum\limits_{k_1+k_2=k}\sum\limits_{\stackrel{i_\bullet\leq j_\bullet}{i_\bullet,j_\bullet\in I_{h_1}}} 
 \frac{(-1)^{\frac{k_1(k_1-1)}{2}+k_1|I_{h_1}|}}{k_1!k_2!}a_{i_1j_1}\cdots a_{i_{k_1}j_{k_1}}\frac{\partial^{k_1}\theta_{I_{h_1}}}{\partial\theta_{j_1}\cdots\partial\theta_{j_{k_1}}}\nonumber \\
 &\quad\quad\quad\quad\quad\quad\quad
  \cdot\partial_{i_1}\cdots\partial_{i_{k_1}}\Big( \sum\limits_{\stackrel{i'_\bullet\leq j'_\bullet}{i'_\bullet,j'_\bullet\in I_{h_1}}}b_{i'_1j'_1}\cdots b_{i'_{k_2}j'_{k_2}}\frac{\partial^{2k_2}\theta_{I_{h_1}}}{\partial\theta_{i'_1}\partial\theta_{j'_1}\cdots\partial\theta_{i'_{k_2}}\partial\theta_{j'_{k_2}}}\Big), \label{eq:omegah1}
\end{align}

\begin{equation}\label{eq:omegah2}
 \frac{1}{l!}{\omega_{h_2}^l(\theta_{I_{h_2}})}
=\sum\limits_{\stackrel{i''_\bullet\leq j''_\bullet}{i''_\bullet,j''_\bullet\in I_{h_2}}}
d_{i''_1j''_1}\cdots d_{i''_{l}j''_{l}}\cdot\frac{\partial^{2l}\theta_{I_{h_2}}}{\partial\theta_{i''_1}\partial\theta_{j''_1}\cdots \partial\theta_{i''_l}\partial\theta_{j''_l}}+\sum_{|I|+|J|=l,|J|\geq 1} f_I \frac{\partial^{2|I|}\theta_{I_{h_2}}}{\partial\theta_{i''_1}\cdots\partial\theta_{i''_{|I|}}}\partial_J.
\end{equation}
The differential operator terms 
$\sum\limits_{|I|+|J|=l,|J|\geq 1} f_I \frac{\partial^{2|I|}\theta_{I_{h_2}}}{\partial\theta_{i''_1}\cdots\partial\theta_{i''_{|I|}}}\partial_J$ 
in \eqref{eq:omegah2} does not contribute to \eqref{eq:matrixmulti}. The comparison of \eqref{eq:matrixmulti} and \eqref{eq:shkformula} is now straightforward.
\end{proof}

Finally, we are ready to prove our main theorem.
\begin{theorem}\label{thm:main}
Define an $R$-module homomorphism 
\begin{align*}
\Phi: \Jac'(W,G) & \to \bigoplus_{g\in G}\Hom(\Delta_1,\Delta_g),\\
 \xi_h & \mapsto [\exp(\eta_{h^{-1}})(\theta_{I_{h^{-1}}})].
 \end{align*}
Then 
\begin{enumerate}
\item $\Phi$ is $G$-equivariant.
\item $\Phi(\xi_{h_1} \bullet \xi_{h_2}) =h_2^{-1}\Phi(\xi_{h_1}) \cup \Phi(\xi_{h_2}).$
\end{enumerate}
\end{theorem}

\begin{proof}
Recall that 
\[ h'\cdot \xi_h= \prod_{l\in I_h} (h'_l)^{ -1} \cdot \xi_h,\]
where the scalar $h'_l \in k^*$ is defined by $h'\cdot x_l=h'_l x_l$.

On a summand $\Hom(\Delta_1,\Delta_{h^{-1}})$,
\[h'\cdot \theta_{I_{h^{-1}}}= \prod_{l\in I_{h^{-1}}} (h'_l)^{-1}\cdot \theta_{I_{h^{-1}}},\] 
but it is clear that $I_{h^{-1}}=I_h$. Furthermore,
\begin{align*}
 h'\cdot (g_{ji}^{h^{-1}} \theta_{I_{h^{-1}}\setminus \{i\}}\partial_j)&= (h'\cdot g_{ji}^{h^{-1}})\rho(h')(\theta_{I_{h^{-1}}\setminus \{i\}}\partial_j)\\
 &=h'\cdot\Big(\frac{(\barW^{h^{-1}}_{j,i}-\barW^{h^{-1}}_{j-1,i})-(\barW^{h^{-1}}_{j,i-1}-\barW^{h^{-1}}_{j-1,i-1})}{(y_j-x_j) (x_i-{h^{-1}}x_i)}\Big)\rho(h')(\theta_{I_{h^{-1}}\setminus \{i\}}\partial_j)\\
&=(h'_j h'_i)^{-1}g_{ji}^{h^{-1}}\cdot\frac{\prod_{l\in I_{h^{-1}}} (h'_l)^{-1}}{(h'_i h'_j)^{-1}}\theta_{I_{h^{-1}}\setminus \{i\}}\partial_j \\
&=\prod_{l\in I_{h^{-1}}} (h'_l)^{-1}\cdot g_{ji}^{h^{-1}}\theta_{I_{h^{-1}}\setminus \{i\}}\partial_j,
\end{align*}
where the third equality comes from $G$-invariance of $W$. The same argument shows that 
\[ h'\cdot (f_{ji}^{h^{-1}}\theta_{I_{h^{-1}}\setminus\{j,i\}})=\prod_{l\in I_{h^{-1}}} (h'_l)^{-1}\cdot f_{ji}^{h^{-1}} \theta_{I_{h^{-1}}\setminus\{j,i\}}.\]
By induction, we conclude that 
\[ h'\cdot \exp(\eta_{h^{-1}})(\theta_{I_{h^{-1}}})=\prod_{l\in I_{h^{-1}}} (h'_l)^{-1}\cdot \exp(\eta_{h^{-1}})(\theta_{I_{h^{-1}}}),\]
and it finishes the proof of the first assertion.

Now we proceed to the second part. Without loss of generality, let
\[ I_{h_1}=I_{h_1^{-1}}=\{1,\cdots,j\}, \quad I_{h_2}=I_{h_2^{-1}}=\{i+1,\cdots,k,k+1,\cdots,j,j+1,\cdots,m\}\]
such that 
$I_{h_1} \cap I_{h_2}=\{i+1,\cdots,k,k+1,\cdots,j\}$ and $h_1\cdot x_l=h_2^{-1}\cdot x_l$ for $k+1\leq l \leq j$. Then
\[ I_{h_1 h_2}=I_{h_2^{-1}h_1^{-1}}=\{1,\cdots,k,j+1,\cdots,m\}\]
and
\[ d_{h_1,h_2}= \frac{d_{h_1}+d_{h_2}-d_{h_1 h_2}}{2}=j-\frac{k+i}{2}=j-k+\frac{k-i}{2}.\]

By Lemma \ref{lem:leadingterm}, the coefficient 
$\big\langle  (h_2^{-1})_* \big(h_2^{-1}\cdot \exp(\eta_{h_1^{-1}})(\theta_{I_{h_1^{-1}}})\big) \cdot \exp(\eta_{h_2^{-1}})(\theta_{I_{h_2^{-1}}}), \theta_{I_{h_2^{-1}h_1^{-1}}}\big\rangle$ 
completely determines the product $[h_2^{-1}\cdot\exp(\eta_{h_1^{-1}})(\theta_{I_{h_1^{-1}}})] \cup [\exp(\eta_{h_2^{-1}})(\theta_{I_{h_2^{-1}}})]$. 
Recall that 
\[\eta_h(\theta_I \partial_J)=  \sum (-1)^{|I|} g^h_{ji} \frac{\partial\theta_I}{\partial\theta_i}\partial_j\partial_J+\sum f^h_{ji}\frac{\partial^2 \theta_I}{\partial\theta_j\partial\theta_i}\partial_J.\]
Decompose $\eta_h$ into following two maps.
\[ \eta_h^{1,1}(\theta_I \partial_J):=  \sum (-1)^{|I|} g^h_{ji} \frac{\partial\theta_I}{\partial\theta_i}\partial_j\partial_J,\]
\[ \eta_h^{2,0}(\theta_I \partial_J):=\sum f^h_{ji}\frac{\partial^2 \theta_I}{\partial\theta_j \partial\theta_i}\partial_J.\]
It is evident that it suffices to consider $\exp(\eta_{h_2^{-1}}^{2,0})(\theta_{I_{h_2^{-1}}})$ only, because differential operator summands(i.e. products of $\partial_{j}$'s) with $\exp(\eta_{h_2^{-1}})(\theta_{I_{h_2^{-1}}})$ still remain after left multiplication of $(h_2^{-1})_* \big(h_2^{-1}\cdot \exp(\eta_{h_1^{-1}})(\theta_{I_{h_1^{-1}}})\big)$, so they do not contribute to the coefficient of $\theta_{I_{h_2^{-1} h_1^{-1}}}$. 
On the other hand, differential operator summands of $(h_2^{-1})_* \big(h_2^{-1}\cdot \exp(\eta_{h_1^{-1}})(\theta_{I_{h_1^{-1}}})\big)$ can act on $\theta$-variables of $\exp(\eta_{h_2^{-1}})(\theta_{I_{h_2^{-1}}})$, so we need to consider whole $(h_2^{-1})_* \big(h_2^{-1}\cdot \exp(\eta_{h_1^{-1}})(\theta_{I_{h_1^{-1}}})\big)$, not only a part of it.

Recall that
\[ \frac{(\eta_{h_2^{-1}}^{2,0})^{k_2}}{k_2!}(\theta_{I_{h_2^{-1}}})= \sum f^{h_2^{-1}}_{i_1 i_2}\cdots f^{h_2^{-1}}_{i_{2k_2-1}i_{2k_2}} \frac{\partial^{2k_2}\theta_{I_{h_2^{-1}}}}{\partial\theta_{i_1}\partial\theta_{i_2}\cdots\partial\theta_{i_{2k_2-1}}\partial\theta_{i_{2k_2}}}.\]
Since 
\[ \theta_{I_{h_1^{-1}}}=\theta_1\cdots \theta_k \theta_{k+1}\cdots\theta_j, \quad \theta_{I_{h_2^{-1}}}=\theta_{i+1}\cdots\theta_k\theta_{k+1}\cdots\theta_j\theta_{j+1}\cdots\theta_m,\quad \theta_{I_{(h_1h_2)^{-1}}}=\theta_1\cdots\theta_k \theta_{j+1}\cdots\theta_m,\]
$\theta_{I_{(h_1h_2)^{-1}}}$-summand of the product comes from differentiating all $\theta_{k+1},\cdots,\theta_j$, differentiating each of $\theta_{i+1},\cdots,\theta_k$ only once in $\theta_{I_{h_1^{-1}}}$ and $\theta_{I_{h_2^{-1}}}$, and then multiplying them. When $\eta_{h_1^{-1}}$ and $\eta_{h_2^{-1}}$ are applied to $\theta_I$, they differentiate $\theta$-variables twice, or differentiate once and give birth to a degree 1 differential operator. We conclude that
\begin{equation}\label{eq:productnonzero}
\big\langle  (h_2^{-1})_* \big(h_2^{-1}\cdot \exp(\eta_{h_1^{-1}})(\theta_{I_{h_1^{-1}}})\big) \cdot \exp(\eta_{h_1^{-1}})(\theta_{I_{h_1^{-1}}}), \theta_{I_{h_2^{-1}h_1^{-1}}}\big\rangle \neq 0
\end{equation}
only if $k_1+k_2=j-k+\frac{k-i}{2}=d_{h_1,h_2}$. In particular, 
\[[h_2^{-1}\cdot\exp(\eta_{h_1^{-1}})(\theta_{I_{h_1^{-1}}})] \cup [\exp(\eta_{h_2^{-1}})(\theta_{I_{h_2^{-1}}})]\neq 0\]
only if $d_{h_1,h_2}\in \Z$.

We calculate the product $\xi_{h_1}\bullet \xi_{h_2}$, namely the coefficient of $\theta_{I_{h_1 h_2}}$ in the following expression
\begin{equation}\label{eq:jacprod}
\frac{\Upsilon\big( (\lfloor H_W(x,h_1\cdot x,x)\rfloor_{h_1 h_2}+\lfloor H_{W,h_1}(x)\rfloor_{h_1 h_2}\otimes 1 + 1\otimes \lfloor H_{W,h_2}(h_1\cdot x)\rfloor_{h_1h_2})^{d_{h_1,h_2}}\otimes {\theta_{I_{h_1}}} \otimes {\theta_{I_{h_2}}}\big)}{d_{h_1,h_2}!}.
\end{equation}

Recall that $H_W(x,y,z)$ is defined by
\[ H_W(x,y,z)=\sum_{1\leq i\leq j \leq n} \nabla_i^{y \to (y,z)}\nabla_j^{x\to (x,y)}(W)\partial_j\otimes\partial_i.\]
We compute each summand more carefully as follows.
\begin{itemize}
\item If $i< j$, then 
\begin{align}
&\lfloor\nabla_i^{y \to (y,z)}\nabla_j^{x\to (x,y)}(W) (x,h_1\cdot x,x)\rfloor_{h_1 h_2}\nonumber\\
=&  \Big\lfloor\frac{\big(W(x_1,\cdots,x_{i-1},h_1x_i,\cdots,h_1x_{j-1},x_j,\cdots,x_n)-W(x_1,\cdots,x_i,h_1x_{i+1},\cdots,h_1x_{j-1},x_j,\cdots,x_n)\big)}{(x_i-h_1x_i)\cdot (h_1x_j-x_j)}\nonumber \\
&\quad -\frac{\big(W(x_1,\cdots,x_{i-1},h_1x_i,\cdots,h_1x_j,x_{j+1},\cdots,x_n)-W(x_1,\cdots,x_i,h_1x_{i+1},\cdots,h_1x_j,x_{j+1},\cdots,x_n)\big)}{(x_i-h_1x_i)\cdot (h_1x_j-x_j)}\Big\rfloor_{h_1h_2}. \label{eq:Hwij}
\end{align}

\item If $i=j$, then
\begin{align}
&\lfloor\nabla_i^{y \to (y,z)}\nabla_i^{x\to (x,y)}(W) (x,h_1\cdot x,x)\rfloor_{h_1h_2}\nonumber\\
=&  \Big\lfloor\frac{\frac{W(z_1,\cdots,z_{i-1},x_i,\cdots,x_n)-W(z_1,\cdots,z_{i-1},y_i,x_{i+1},\cdots,x_n)}{x_i-y_i}  -\frac{W(z_1,\cdots,z_{i-1},x_i,\cdots,x_n)-W(z_1,\cdots,z_i,x_{i+1},\cdots,x_n)}{x_i-z_i}}{y_i-z_i}|_{y=h_1 \cdot x, z=x}\Big\rfloor_{h_1h_2}.\label{eq:Hwii}
\end{align}


\end{itemize}

\begin{lemma}\label{lem:gij}
\[ \eqref{eq:Hwij}= \lfloor(1\times h_2)\cdot h_2^{-1}\cdot g_{ij}^{h_1^{-1}}|_{y=(h_1h_2)^{-1}x}\rfloor_{h_1 h_2}, \]
\[\eqref{eq:Hwii}= \lfloor(1\times h_2)\cdot h_2^{-1}\cdot g_{ii}^{h_1^{-1}}|_{y=(h_1h_2)^{-1}x}\rfloor_{h_1 h_2}.\]
\end{lemma}

\begin{proof}[Proof of the Lemma]
\begin{align*}
&\lfloor(1\times h_2)\cdot h_2^{-1}\cdot g_{ij}^{h_1^{-1}}|_{y=h_2^{-1}h_1^{-1}x}\rfloor_{h_1h_2} \\
=& \lfloor(h_2^{-1}\times 1)\cdot g_{ij}^{h_1^{-1}}|_{y=(h_1h_2)^{-1}x}\rfloor_{h_1h_2} \\
=&\Big\lfloor \frac{1}{(h_2^{-1}h_1^{-1}x_i-h_2^{-1}x_i)(h_2^{-1}x_j-h_2^{-1}h_1^{-1}x_j)}\\
&\cdot\big(W(h_2^{-1}h_1^{-1}x_1,\cdots,h_2^{-1}h_1^{-1}x_i,h_2^{-1}x_{i+1},\cdots,h_2^{-1}x_j,h_2^{-1}h_1^{-1}x_{j+1},\cdots,h_2^{-1}h_1^{-1}x_n)\\
&\quad-W(h_2^{-1}h_1^{-1}x_1,\cdots,h_2^{-1}h_1^{-1}x_{i-1},h_2^{-1}x_i,\cdots,h_2^{-1}x_j,h_2^{-1}h_1^{-1}x_{j+1},\cdots,h_2^{-1}h_1^{-1}x_n) \\
&\quad-W(h_2^{-1}h_1^{-1}x_1,\cdots,h_2^{-1}h_1^{-1}x_i,h_2^{-1}x_{i+1},\cdots,h_2^{-1}x_{j-1},h_2^{-1}h_1^{-1}x_{j},\cdots,h_2^{-1}h_1^{-1}x_n)\\
&\quad+W(h_2^{-1}h_1^{-1}x_1,\cdots,h_2^{-1}h_1^{-1}x_{i-1},h_2^{-1}x_i,\cdots,h_2^{-1}x_{j-1},h_2^{-1}h_1^{-1}x_{j},\cdots,h_2^{-1}h_1^{-1}x_n)\big)\Big\rfloor_{h_1h_2} \\
=&\Big\lfloor \frac{1}{(x_i-h_1x_i)(h_1x_j-x_j)}\\
&\cdot\big(W(x_1,\cdots,x_i,h_1x_{i+1},\cdots,h_1x_j,x_{j+1},\cdots,x_n)
-W(x_1,\cdots,x_{i-1},h_1x_i,\cdots,h_1x_j,x_{j+1},\cdots,x_n) \\
&\quad-W(x_1,\cdots,x_i,h_1x_{i+1},\cdots,h_1x_{j-1},x_{j},\cdots,x_n)
+W(x_1,\cdots,x_{i-1},h_1x_i,\cdots,h_1x_{j-1},x_{j},\cdots,x_n)\big)\Big\rfloor_{h_1h_2}.
\end{align*}
The last identity is due to the following: for any polynomial $f(x_1,\cdots,x_n)$,
\begin{equation}\label{eq:jacrelation}
 \lfloor f(x_1,\cdots,x_n) \rfloor_{h_1h_2}= \lfloor f(h_1h_2x,\cdots,h_1h_2x_n)\rfloor_{h_1h_2}.
 \end{equation}
The proof of the second identity is similar so omitted.
\end{proof}

Next we consider 
\[ H_{W,h_1}=\sum_{i,j\in I_{h_1},i<j}\frac{1}{1-(h_1)_i}\nabla_i^{x \to (x,x^{h_1})}\nabla_j^{x \to (x,h_1\cdot x)}(W)\partial_i \partial_j.\]
Again, a summand can be written as
\begin{align}
&\Big\lfloor \frac{1}{1-(h_1)_i}\nabla_i^{x \to (x,x^{h_1})}\nabla_j^{x \to (x,h_1\cdot x)}(W) \Big\rfloor_{h_1h_2}\nonumber\\
=& \Big\lfloor \frac{1}{(1-(h_1)_i)(x_i^{h_1}-x_i)(h_1x_j-x_j)}\nonumber\\
&\cdot \big(W(x_1^{h_1},\cdots,x_i^{h_1},h_1x_{i+1},\cdots,h_1x_j,x_{j+1},\cdots,x_n)
-W(x_1^{h_1},\cdots,x_i^{h_1},h_1x_{i+1},\cdots,h_1x_{j-1},x_j,\cdots,x_n)\nonumber\\
&\quad-W(x_1^{h_1},\cdots,x_{i-1}^{h_1},h_1x_{i},\cdots,h_1x_j,x_{j+1},\cdots,x_n)
+W(x_1^{h_1},\cdots,x_{i-1}^{h_1},h_1x_{i},\cdots,h_1x_{j-1},x_j,\cdots,x_n) \big)\Big\rfloor_{h_1h_2}\nonumber \\  
=&\Big\lfloor \frac{1}{(h_1x_i-x_i)(h_1x_j-x_j)}\nonumber\\
&\cdot \big(W(x_1^{h_1},\cdots,x_i^{h_1},h_1x_{i+1},\cdots,h_1x_j,x_{j+1},\cdots,x_n)
-W(x_1^{h_1},\cdots,x_i^{h_1},h_1x_{i+1},\cdots,h_1x_{j-1},x_j,\cdots,x_n)\nonumber\\
&\quad-W(x_1^{h_1},\cdots,x_{i-1}^{h_1},h_1x_{i},\cdots,h_1x_j,x_{j+1},\cdots,x_n)
+W(x_1^{h_1},\cdots,x_{i-1}^{h_1},h_1x_{i},\cdots,h_1x_{j-1},x_j,\cdots,x_n) \big)\Big\rfloor_{h_1h_2}. \label{eq:Hh1}
\end{align}
For the last identity, we used $x_i^{h_1}=0$. We also observe that
\begin{align}
\lfloor H_{W,h_2}(h_1\cdot x)\rfloor_{h_1h_2}=&\Big\lfloor \frac{1}{1-(h_2)_i}\big(\nabla_i^{x \to (x,x^{h_2})}\nabla_j^{x \to (x,h_2\cdot x)}(W)\big)(h_1\cdot x) \Big\rfloor_{h_1h_2}\nonumber\\
=& \Big\lfloor \frac{1}{(h_1h_2x_i-h_1x_i)(h_1h_2x_j-h_1x_j)}\nonumber\\
&\cdot \big(W(h_1x_1^{h_2},\cdots,h_1x_i^{h_2},h_1h_2x_{i+1},\cdots,h_1h_2x_j,h_1x_{j+1},\cdots,h_1x_n)\nonumber\\
&\quad-W(h_1x_1^{h_2},\cdots,h_1x_i^{h_2},h_1h_2x_{i+1},\cdots,h_1h_2x_{j-1},h_1x_j,\cdots,h_1x_n)\nonumber\\
&\quad-W(h_1x_1^{h_2},\cdots,h_1x_{i-1}^{h_2},h_1h_2x_{i},\cdots,h_1h_2x_j,h_1x_{j+1},\cdots,h_1x_n)\nonumber\\
&\quad+W(h_1x_1^{h_2},\cdots,h_1x_{i-1}^{h_2},h_1h_2x_{i},\cdots,h_1h_2x_{j-1},h_1x_j,\cdots,h_1x_n) \big)\Big\rfloor_{h_1h_2} \nonumber\\  
=&\Big\lfloor \frac{1}{(x_i-h_2^{-1}x_i)(x_j-h_2^{-1}x_j)}\nonumber\\
&\cdot \big(W(h_2^{-1}x_1^{h_2},\cdots,h_2^{-1}x_i^{h_2},x_{i+1},\cdots,x_j,h_2^{-1}x_{j+1},\cdots,h_2^{-1}x_n)\nonumber\\
&\quad-W(h_2^{-1}x_1^{h_2},\cdots,h_2^{-1}x_i^{h_2},x_{i+1},\cdots,x_{j-1},h_2^{-1}x_j,\cdots,h_2^{-1}x_n)\nonumber\\
&\quad-W(h_2^{-1}x_1^{h_2},\cdots,h_2^{-1}x_{i-1}^{h_2},x_{i},\cdots,x_j,h_2^{-1}x_{j+1},\cdots,h_2^{-1}x_n)\nonumber\\
&\quad+W(h_2^{-1}x_1^{h_2},\cdots,h_2^{-1}x_{i-1}^{h_2},x_{i},\cdots,x_{j-1},h_2^{-1}x_j,\cdots,h_2^{-1}x_n) \big)\Big\rfloor_{h_1h_2}\nonumber\\
=&\Big\lfloor \frac{1}{(x_i-h_2^{-1}x_i)(x_j-h_2^{-1}x_j)}\nonumber\\
&\cdot \big(W(x_1^{h_2},\cdots,x_i^{h_2},x_{i+1},\cdots,x_j,h_2^{-1}x_{j+1},\cdots,h_2^{-1}x_n)\nonumber\\
&\quad-W(x_1^{h_2},\cdots,x_i^{h_2},x_{i+1},\cdots,x_{j-1},h_2^{-1}x_j,\cdots,h_2^{-1}x_n)\nonumber\\
&\quad-W(x_1^{h_2},\cdots,x_{i-1}^{h_2},x_{i},\cdots,x_j,h_2^{-1}x_{j+1},\cdots,h_2^{-1}x_n)\nonumber\\
&\quad+W(x_1^{h_2},\cdots,x_{i-1}^{h_2},x_{i},\cdots,x_{j-1},h_2^{-1}x_j,\cdots,h_2^{-1}x_n) \big)\Big\rfloor_{h_1h_2}.\label{eq:Hh2}
\end{align}
The second identity comes from 
\[ \lfloor f(x_1,\cdots,x_n)\rfloor_{h_1h_2}=\lfloor f(h_2^{-1}h_1^{-1}x_1,\cdots,h_2^{-1}h_1^{-1}x_n)\rfloor_{h_1h_2}\]
for any polynomial $f(x_1,\cdots,x_n)$.
 The following lemma can be proved similarly as Lemma \ref{lem:gij}.
\begin{lemma}\label{lem:fij}
\[ \eqref{eq:Hh1}= \lfloor(1\times h_2)\cdot h_2^{-1}\cdot f_{ij}^{h_1^{-1}}\rfloor_{h_1 h_2}  ,\quad \eqref{eq:Hh2}= \lfloor f_{ij}^{h_2^{-1}}\rfloor_{h_1h_2}. \]
\end{lemma}

By Lemma \ref{lem:gij} and \ref{lem:fij}, we can rewrite \eqref{eq:jacprod} as
\begin{align*}
 \frac{1}{d_{h_1,h_2}!} &\cdot
\Upsilon\Big( \big(\sum_{i,j\in I_{h_1},i\leq j}(\lfloor(1\times h_2)\cdot h_2^{-1}\cdot g_{ij}^{h_1^{-1}}|_{y=(h_1h_2)^{-1}x}\rfloor_{h_1 h_2}\cdot\partial_j\otimes\partial_i \\
&\quad\quad+\sum_{i,j\in I_{h_1},i<j}\lfloor(1\times h_2)\cdot h_2^{-1}\cdot f_{ij}^{h_1^{-1}}\rfloor_{h_1 h_2}\cdot \partial_i \partial_j\otimes 1 )\\
&\quad\quad+ \sum_{i,j\in I_{h_2},i<j}  \lfloor f_{ij}^{h_2^{-1}}\rfloor_{h_1h_2}\cdot 1\otimes \partial_i \partial_j \big)^{d_{h_1,h_2}}\otimes {\theta_{I_{h_1}}} \otimes {\theta_{I_{h_2}}}\Big).
\end{align*}
The proof is finally finished by applying Lemma \ref{lem:mf-jac-comparison}.
\end{proof}

\begin{corollary}
$\displaystyle\bigoplus_{g\in G} \Hom(\Delta_1,\Delta_g)$ is braided super-commutative in the following sense:
\begin{equation*}
[\exp(\eta_g)(\theta_{I_g})] \cup [\exp(\eta_h)(\theta_{I_h})] = (g\cdot [\exp(\eta_h)(\theta_{I_h})]\big) \cup [\exp(\eta_g)(\theta_{I_g})].
\end{equation*}
\end{corollary}

\section{Application: Lagrangian Floer theory}\label{sec:application}
The original motivation of this work was to enhance the understanding of Floer theory for representing objects of homological mirror functors. We briefly review how we obtain a mirror algebraic function (often called a {\em potential}) of a symplectic manifold. Let $M$ be a symplectic manifold and $\bL\subset M$ be a Lagrangian submanifold. 
Suppose that $\bL$ is weakly unobstructed, i.e. there is a space $\mathcal{MC}(\bL)$ of cochains of $\bL$ such that the $\AI$-structure $\{m_k\}_{k\geq 0}$ can be deformed by elements of $\mathcal{MC}(\bL)$, so that the obstruction for $m_1^2  =0$ vanishes. Elements of $\mathcal{MC}(\bL)$ are called {\em weak bounding cochains}. The potential $W_\bL:\MC(\bL) \to k$ is given by counting holomorphic discs whose boundaries are on $\bL$ and decorated by weak bounding cochains. Suppose that $0$ is the critical value of $W_\bL$. Then we can define an $\AI$-functor
\[ \CF^\bL: Fuk(M) \to MF(W_\bL)\]
by the curved Yoneda functor induced by $\bL$. (For detail see e.g. \cite{CHL1}, in which $\CF^\bL$ was named as {\em localized mirror functor.})

Our primary examples are orbifold spheres $\PP^1_{a,b,c}$ which are quotients of Riemann surfaces by finite groups. On $\PP^1_{a,b,c}$, there is an immersed Lagrangian $\bar{\bL}$ which is called the Seidel Lagrangian (see \cite{Sei}). It was proved in \cite{CHL1} that $\bar{\bL}$ is weakly unobstructed, and the localized mirror functor $\CF^{\bar{\bL}}$ is an equivalence if $\PP^1_{a,b,c}$ is elliptic or hyperbolic. Furthermore, $\MC(\bar{\bL})$ consists of linear combinations of three odd degree immersed generators. Hence the potential $W_{\bar{\bL}}$ is an element of $k\llbracket x,y,z \rrbracket$. If $\PP^1_{a,b,c}$ is spherical or elliptic, then $W_{\bar{\bL}}$ is in $k[x,y,z]$.

Let us focus on the following three elliptic orbispheres (see Figure \ref{fig:236333} and Figure \ref{fig:244}).
\[\PP^1_{2,3,6}=T^2/ (\Z/6\Z), \quad \PP^1_{2,4,4}=T^2/ (\Z/4\Z), \quad \PP^1_{3,3,3}=T^2/(\Z/3\Z).\]
$\PP^1_{2,3,6}$ and $\PP^1_{3,3,3}$ are quotients of $\C/(\Z \oplus e^{2\pi i/3}\Z)$ by $\Z/6\Z$ and $\Z/3\Z$ respectively, and $ \PP^1_{2,4,4}$ is a quotient of $\C/(\Z \oplus i\Z)$ by $\Z/4\Z$. In Figure \ref{fig:236333} and Figure \ref{fig:244}, odd degree immersed generators are denoted by $X_1$, $X_2$ and $X_3$, shaded regions are fundamental domains of orbifolds, black dots are orbifold points with orders of isotropy groups specified, and dotted curves on tori are unions of embedded circles which are preimages of Seidel Lagrangians (which are dotted curves inside fundamental domains of orbifolds) via quotient maps.
 
\begin{figure}
\includegraphics[height=1.9in]{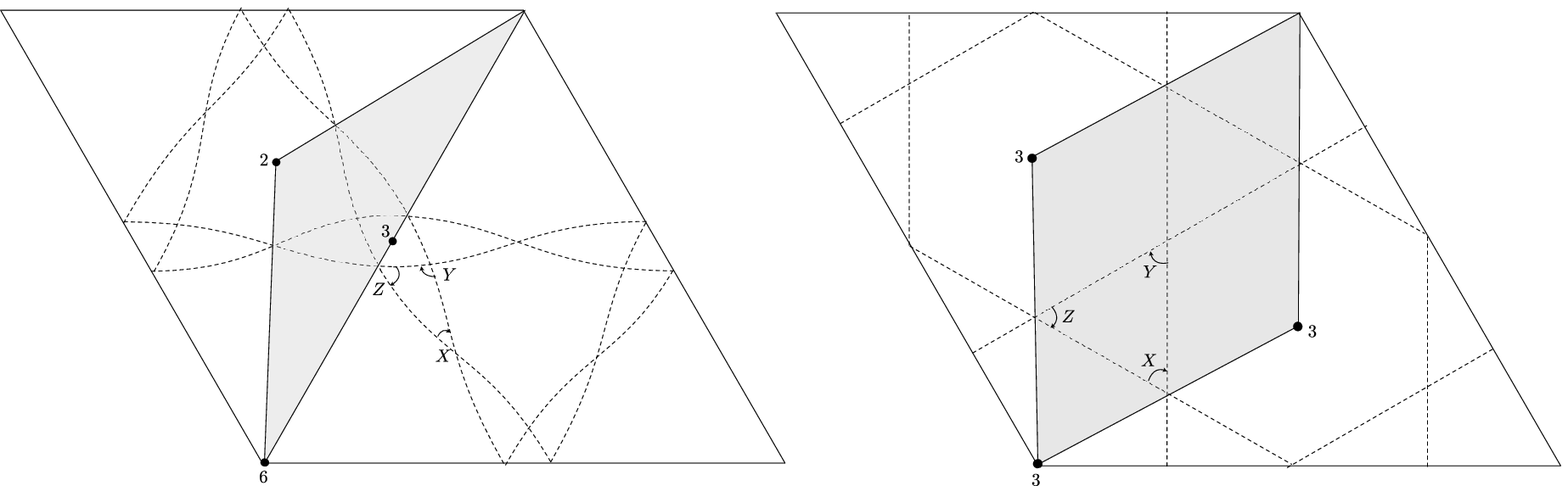}
\caption{$\PP^1_{2,3,6}$ and $\PP^1_{3,3,3}$.}
\label{fig:236333}
\end{figure}

\begin{figure}
\includegraphics[height=1.7in]{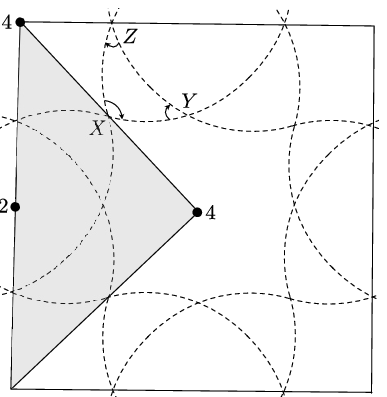}
\caption{$\PP^1_{2,4,4}$.}
\label{fig:244}
\end{figure}
 Let $W_{2,3,6}$, $W_{3,3,3}$ and $W_{2,4,4}$ be potentials from Seidel Lagrangians in $\PP^1_{2,3,6}$, $\PP^1_{3,3,3}$ and $ \PP^1_{2,4,4}$ respectively. They were computed in \cite{CHL1,CHKL} as follows.
 \begin{theorem}[\cite{CHL1,CHKL}]
 \begin{align*}
W_{2,3,6}&= q^6 x_1^2-qx_1x_2x_3+a_1x_2^3+a_2x_3^6+a_3 x_2^2 x_3^2+a_4x_2 x_3^4,\\
W_{3,3,3}&= b_1(x_1^3+x_2^3+x_3^3)-b_2x_1x_2x_3, \\
W_{2,4,4}&=q^6 x_1^2-qx_1x_2x_3+c_1x_2^4+c_2x_3^4+c_3x_2^2 x_3^2,
\end{align*}
where $a_1,a_2,a_3,a_4,b_1,b_2,c_1,c_2,c_3$ are elements in the Novikov field $k=\Lambda$ ($q$ is the Novikov variable). We omit their precise expressions.
\end{theorem}

We state a technical assumption which we need.

\begin{assumption}\label{as:floergen}
\begin{itemize}
\item $\MC(W_{\bar{\bL}})$ consists of linear combinations of $\{e_1,\cdots,e_n\}\subset CF^1(\bar{\bL},\bar{\bL})$.
\item Let $R$ be the function ring on $\MC(W_{\bar{\bL}})$, i.e. $R=k[x_1,\cdots,x_n]$. The Floer complex $CF(\bar{\bL},\bar{\bL})\otimes R$ is generated by following elements
\[ e_I^{\bar{b}}:=m_2^{\bar{b}}(e_{i_1},m_2^{\bar{b}}(e_{i_2},m_2^{\bar{b}}(\cdots,e_{i_k}),\cdots)),\]
where $m_2^{\bar{b}}$ is the binary $\AI$-operation deformed by
\[ {\bar{b}}=x_1e_1+\cdots+x_n e_n.\]
\end{itemize}
\end{assumption}
It is reasonable to predict that any Lagrangian submanifold which represents a mirror equivalence satisfies this assumption and it will be studied elsewhere. Now we recall some results.
\begin{theorem}[\cite{CL}]
\begin{itemize}
\item Any Seidel Lagrangian $\bar{\bL}\subset \PP^1_{a,b,c}$ satisfy Assumption \ref{as:floergen}.
\item If a weakly unobstructed Lagrangian $(\bL,b) \subset M$ satisfies Assumption \ref{as:floergen}, then 
\[ HF\big((\bL,b),(\bL,b)\big)\otimes R \cong \Jac(W_\bL).\]
\item If $\bar{\bL} \subset M/G$ satisfies Assumption \ref{as:floergen} and $\bL\subset M$ is the preimage of $\bar{\bL}$ via quotient, then 
\[ HF\big((\bL,b),(\bL,b)\big)\otimes R \cong \Hom_{MF_{1\times \HG}(W_{\bar{\bL}}(y)-W_{\bar{\bL}}(x))}(\Delta_{W_{\bar{\bL}}}^{\HG\times\HG},\Delta_{W_{\bar{\bL}}}^{\HG\times \HG}),\]
where $b$ is a lift of $\bar{b}$.
\end{itemize}
\end{theorem}
Combining this with Theorem \ref{thm:main}, we can find a quite interesting algebraic structure on $\bL$ as follows: 
\begin{theorem}\label{thm:floerjac}
For each pair $(T^2,G)$ where $G$ is one of $\Z/6\Z$, $\Z/4\Z$ or $\Z/3\Z$, consider the immersed Lagrangian submanifold $\bL \subset T^2$ together with a cochain 
\[b=x_1(\sum_{g\in G} g\cdot X_1)+x_2(\sum_{g\in G} g\cdot X_2)+x_3(\sum_{g\in G} g\cdot X_3) \in CF^1(\bL,\bL)\otimes R.\]
Then we have an algebra isomorphism
\[HF\big((\bL,b),(\bL,b)\big)\otimes R \cong \Jac'(W_{\bar{\bL}},\HG).\] 
\end{theorem}
Theorem \ref{thm:floerjac} implies that $HF\big((\bL,b),(\bL,b)\big)\otimes R$ is a braided super-commutative algebra.

\begin{remark}
Any twisted Jacobian algebra has a sector decomposition indexed by group elements, so $HF\big((\bL,b),(\bL,b)\big)\otimes R$ also has such a decomposition. It is given by an $\AI$ quasi-isomorphism as follows.
\begin{align}
 \Phi: CF\big((\bar{\bL},\bar{b}),(\bar{\bL},\bar{b})\big) \rtimes \HG \;\simeq\; & CF\big((\bL,b),(\bL,b)\big),\label{eq:Phi}\\
 v\otimes \chi \; \mapsto \;& \sum_{g\in G} \chi(g^{-1})(g\cdot v). \nonumber
\end{align}
The $\AI$-structure on $ CF\big((\bar{\bL},\bar{b}),(\bar{\bL},\bar{b})\big) \rtimes \HG$ is addressed in \cite[Section 6]{CL}. In the following example we will use this identification to describe Floer cohomology elements for $(\bL,b) \subset \PP^1_{2,3,6}$.
\end{remark}

\subsection{An example}
In the remainder, let us illustrate the structure of $\Jac'(W_{2,3,6},\Z/6\Z)$. For convenience let $W:=W_{2,3,6}$.
Denote elements of $\Z/6\Z$ by $\{\rho_0,\rho_1,\cdots,\rho_5\}$, and let $\Z/6\Z$ act on $k[x_1,x_2,x_3]$ by
\[ \rho_l \cdot x_1= e^{\pi li} x_1, \quad \rho_l \cdot x_2=e^{2\pi li/3}x_2 ,\quad \rho_l\cdot x_3=e^{\pi li/3}x_3.\]
Then we have
\[ I_{\rho_0}=\emptyset,\; I_{\rho_1}=I_{\rho_5}=\{1,2,3\},\; I_{\rho_2}=I_{\rho_4}=\{2,3\},\; I_{\rho_3}=\{1,3\} \] 
and
\[ W^{\rho_0}=W,\;\; W^{\rho_1}=W^{\rho_5}=0\in k,\] 
\[W^{\rho_2}=W^{\rho_4}=q^6 x_1^2\in k[x_1],\;\; W^{\rho_3}=c_{x_2}(q)x_2^3\in k[x_2].\]
Therefore, the underlying module structure of $\Jac'(W,\Z/6\Z)$ is as follows.
\[ \Jac'(W,\Z/6\Z)\cong\Jac(W)\oplus k \xi_{\rho_1}\oplus k \xi_{\rho_2} \oplus \big(k[x_2]/(x_2^2)\big) \xi_{\rho_3} \oplus k \xi_{\rho_4}\oplus k \xi_{\rho_5}.\]

It already tells a lot about the structure of $HF\big((\bL,b),(\bL,b)\big)\otimes R$. Its summand generated by the unit $1_{\bL}$ is isomorphic to $\Jac(W)$. It has 5 more sectors indexed by $\rho_1,\cdots,\rho_5 \in \Z/6\Z$. 
To describe those Floer cohomology elements more precisely, we refer to some previous works. 
Let us take an example $\xi_{\rho_2}$, where $I_{\rho_2}=\{2,3\}$. By combining Proposition 4.8 and Proposition 9.1 of \cite{CL}, we conclude that $\xi_{\rho_2}$ corresponds to the Floer cohomology element
\[\Phi(m_2(X_2,X_3))=\Phi(c\cdot \bar{X_1})+\alpha \in HF\big((\bL,b),(\bL,b)\big)\otimes R\] 
for some $c\in k$, where $\alpha$ is the sum of lower degree elements. $\xi_{\rho_1}$ corresponds to
\[ \Phi(m_2(X_1,m_2(X_2,X_3))=\Phi(c'\cdot pt_\bL)+\alpha' \in HF\big((\bL,b),(\bL,b)\big)\otimes R\] 
for some $c'$ and $\alpha'$. Similarly for other generators, $\xi_{\rho_4}$ correspond to Floer cocycles whose leading order cochains are $\Phi(\bar{X_1})$, and $\xi_{\rho_3}$ corresponds to the Floer cocycle whose leading order cochain is $\Phi(\bar{X_2})$.
 
We can check that $d_{\rho_0,\rho_i}=0$ ($i=1,2,3$), $d_{\rho_1,\rho_5}=d_{\rho_5,\rho_1}=3$, $d_{\rho_2,\rho_4}=d_{\rho_4,\rho_2}=d_{\rho_3,\rho_3}=2$ are only integers among $d_{g,h}$ for $g,h\in \Z/6\Z$. By $d_{\rho_0,\rho_i}=0$, we have a few straightforward computations: \[ \xi_{\rho_0} \bullet \xi_{\rho_i}=\xi_{\rho_i}.\]

Computation of $\xi_{\rho_1}\bullet \xi_{\rho_5}$, $\xi_{\rho_3}\bullet \xi_{\rho_3}$ and $\xi_{\rho_2}\bullet \xi_{\rho_4}$ comes from the formula \eqref{eq:twjacprod}. By braided super-commutativity, we easily deduce that 
\begin{equation} \label{eq:commutativity}
 \xi_{\rho_1}\bullet \xi_{\rho_5}=-\xi_{\rho_5}\bullet \xi_{\rho_1}, \quad 
\xi_{\rho_2}\bullet \xi_{\rho_4}= \xi_{\rho_4}\bullet (\rho_4)^{-1}\xi_{\rho_2}=\xi_{\rho_4}\bullet \xi_{\rho_2}.
\end{equation}
Actually, we can identify $\Z/6\Z$-invariant elements of $\Jac'(W,\Z/6\Z)$ as images of elements in $H^*(T^2)$ via {\em Kodaira-Spencer map}, which is a ring isomorphism. In particular, two odd degree elements $\xi_{\rho_1}$ and $\xi_{\rho_5}$ are $\Z/6\Z$-invariant elements and correspond to degree 1 cohomology classes of $T^2$ via Kodaira-Spencer map, so it is straightforward that their product is nontrivial and skew-symmetric. We refer to \cite{CL} for the detail.

Let us illustrate the computation of $\xi_{\rho_2}\bullet \xi_{\rho_4}$ which does not exist in the orbifold Jacobian ring. To use the formula \eqref{eq:twjacprod}, we first compute the following polynomials which are given by applying difference operators twice to $W$.
\begin{align*}
\nabla_{11}W:=\nablayz{1}\nablaxy{1} W=&q^6,\\
\nabla_{12}W:=\nablayz{1}\nablaxy{2} W=& -qx_3, \\
\nabla_{22}W:=\nablayz{2}\nablaxy{2}W=&a_1(z_2+y_2+x_2)+a_3x_3^2,\\
\nabla_{13}W:=\nablayz{1}\nablaxy{3}W=& -qy_2, \\
\nabla_{23}W:=\nablayz{2}\nablaxy{3}W=& -qy_1+a_3(y_3+z_3)(z_2+y_2)+a_4(y_3+x_3)(y_3^2+x_3^2),\\
\nabla_{33}W:=\nablayz{3}\nablaxy{3}W=& a_2 \frac{z_3^5-y_3^5}{z_3-y_3}+a_2x_3\frac{z_3^4-y_3^4}{z_3-y_3}+(a_2x_3^2+a_4y_2)\frac{z_3^3-y_3^3}{z_3-y_3}+(a_2x_3^3+a_4y_2x_3)\frac{z_3^2-y_3^2}{z_3-y_3}\\
&+a_2x_3^4+a_3y_2^2+a_4y_2x_3^2.
\end{align*}
Then we deduce that
\begin{align*}
 \sigma_{\rho_2,\rho_4}=&\frac{1}{2}\Upsilon \Big(\big(\lfloor \nabla_{22}W(x,\rho_2 x,x) \rfloor\partial_2\otimes\partial_2
 +\lfloor\nabla_{23}W(x,\rho_2 x,x) \rfloor\partial_3\otimes\partial_2 
 +\lfloor\nabla_{33}W(x,\rho_2 x,x) \rfloor\partial_3\otimes\partial_3\\
 &\quad\quad +\big\lfloor\frac{\nabla_{23}W(x,\rho_2 x,x^{\rho_2})}{1-(\rho_2)_2} \big\rfloor\partial_2\partial_3 \otimes 1
 +\big\lfloor\frac{\nabla_{23}W(\rho_2 x,\rho_4\rho_2 x,\rho_2(x^{\rho_4}))}{1-(\rho_4)_2} \big\rfloor1\otimes\partial_2\partial_3\big)^2 \otimes \theta_2\theta_3 \otimes \theta_2\theta_3\Big)\\
 =& \frac{1}{2}\Upsilon \Big(\big( \lfloor\nabla_{22}W(x,\rho_2 x,x)\rfloor\partial_2\otimes\partial_2 
 +\lfloor\nabla_{23}W(x,\rho_2 x,x)\rfloor\partial_3\otimes\partial_2 
 +\lfloor\nabla_{33}W(x,\rho_2 x,x)\rfloor\partial_3\otimes\partial_3\big)^2 \otimes  \theta_2\theta_3 \otimes \theta_2\theta_3\Big)\\
 =& -\Upsilon\big((\lfloor\nabla_{22}W(x,\rho_2 x,x)\cdot\nabla_{33}W(x,\rho_2 x,x)\rfloor\partial_2\partial_3 \otimes \partial_2\partial_3)
 \otimes \theta_2\theta_3 \otimes \theta_2\theta_3\big)\\
 =&-\lfloor\nabla_{22}W(x,\rho_2 x,x)\cdot\nabla_{33}W(x,\rho_2 x,x)\rfloor \in \Jac(W).
\end{align*}
We have
\begin{align*}
\nabla_{22}W(x,\rho_2 x,x)=&a_1(x_2+\rho_2 x_2)+a_1x_2+a_3x_3^2=(1-(\rho_2)_3)a_1x_2+a_3x_3^2,\\
\nabla_{33}W(x,\rho_2 x,x)=& \big( \frac{1-(\rho_2)_3^5}{1-(\rho_2)_3}+\cdots+\frac{1-(\rho_2)_3}{1-(\rho_2)_3}\big) a_2x_3^4
+\big( \frac{1-(\rho_2)_3^3}{1-(\rho_2)_3}+\cdots+\frac{1-(\rho_2)_3}{1-(\rho_2)_3}\big)a_4(\rho_2 x_2)x_3^2
+a_3(\rho_2 x_2)^2\\
=& \frac{6}{1-(\rho_2)_3}a_2x_3^4+\frac{3}{1-(\rho_2)_3}a_4(\rho_2 x_2)x_3^2+a_3 (\rho_2x_2)^2,
\end{align*}
so
\begin{align*}
\nabla_{22}(x,\rho_2 x,x)\cdot \nabla_{33}(x,\rho_2 x,x) =& \big((1-(\rho_2)_3)a_1x_2+a_3x_3^2\big) \cdot \big(\frac{6}{1-(\rho_2)_3}a_2x_3^4+\frac{3}{1-(\rho_2)_3}a_4(\rho_2 x_2)x_3^2+a_3 (\rho_2x_2)^2\big)\\
=&a_1a_3(1-(\rho_2)_3)(\rho_2)_2^2 x_2^3+ (a_3^2(\rho_2)_2+3a_1a_4)(\rho_2)_2x_2^2x_3^2\\
&+ \big(\frac{3a_3a_4(\rho_2)_2}{1-(\rho_2)_3}+6a_1a_2\big)x_2x_3^4+\frac{6}{1-(\rho_2)_3}a_2a_3x_3^6.
\end{align*}
Since the Jacobian ideal of $W$ is
\[ \partial W=\big(2q^6x_1-qx_2x_3,-qx_1x_3+3a_1x_2^2+2a_3x_2x_3^2+a_4x_3^4,-qx_1x_2+6a_2x_3^5+2a_3x_2^2x_3+4a_4x_2x_3^3\big),\]
 we readily check that $\sigma_{\rho_2,\rho_4}=\lfloor \nabla_{22}(x,\rho_2 x,x)\cdot \nabla_{33}(x,\rho_2 x,x)\rfloor$ is nonzero in $\Jac(W)$.

Likewise, we can also compute 
\begin{align*}
\sigma_{\rho_3,\rho_3}=&\frac{1}{2}\Upsilon \Big( \big(q^6 \partial_1\otimes\partial_1 + \big\lfloor\frac{6}{1-(\rho_3)_3}a_2x_3^4+\frac{3}{1-(\rho_2)_3}a_4(\rho_3 x_2)x_3^2+a_3(\rho_3 x_2)^2\big\rfloor\partial_3\otimes\partial_3\big)^2\otimes \theta_1\theta_3\otimes \theta_1\theta_3\Big) \\
=& -\Upsilon\Big(\big(\big\lfloor q^6 \big(\frac{6}{1-(\rho_3)_3}a_2x_3^4+\frac{3}{1-(\rho_2)_3}a_4(\rho_3 x_2)x_3^2+a_3(\rho_3 x_2)^2\big\rfloor
 \partial_1\partial_3\otimes\partial_1\partial_3\big) \otimes \theta_1\theta_3 \otimes \theta_1\theta_3\Big)\\
 =& \big\lfloor- q^6 \big(\frac{6}{1-(\rho_3)_3}a_2x_3^4+\frac{3}{1-(\rho_2)_3}a_4(\rho_3 x_2)x_3^2+a_3(\rho_3 x_2)^2\big)\big\rfloor
\end{align*}
and it is also straightforward to verify that $\sigma_{\rho_3,\rho_3}$ is nonzero in $\Jac(W)$. 

Summarizing in terms of Floer cohomology, we have the following.
\begin{theorem}
Let $\bar{\bL} \subset \PP^1_{2,3,6}$. Let $(\bL,b)$ be as above. Then we have a decomposition
\[ HF\big((\bL,b),(\bL,b)\big)\otimes R= \bigoplus_{\rho_i \in \Z/6\Z} HF(\bL,\bL)_{\rho_i},\]
where
\begin{align*}
HF(\bL,\bL)_{\rho_0}\cong & \Jac(W)\cdot 1_\bL,\\
HF(\bL,\bL)_{\rho_1}\cong & k\cdot \langle pt_\bL+ aX+bY+cZ\rangle,\\
HF(\bL,\bL)_{\rho_5}\cong & k\cdot \langle pt_\bL+a'X+b'Y+c'Z\rangle,\\
HF(\bL,\bL)_{\rho_2}\cong & k\cdot \langle \bar{X}+d\cdot 1_\bL\rangle,\\
HF(\bL,\bL)_{\rho_4}\cong & k\cdot \langle \bar{X}+d'\cdot 1_\bL\rangle,\\
HF(\bL,\bL)_{\rho_3}\cong & \big(k[x_2]/(x_2^2)\big)\cdot \langle \bar{Y}+l\cdot 1_\bL\rangle,
\end{align*}
and we have nontrivial multiplications between following pairs of components:
\begin{align*}
HF(\bL,\bL)_{\rho_0} \otimes HF(\bL,\bL)_{\rho_i} \to & HF(\bL,\bL)_{\rho_i},\\
 HF(\bL,\bL)_{\rho_1}\otimes HF(\bL,\bL)_{\rho_5} \to & HF(\bL,\bL)_{\rho_0}, \\
 HF(\bL,\bL)_{\rho_2}\otimes HF(\bL,\bL)_{\rho_4} \to & HF(\bL,\bL)_{\rho_0}, \\
 HF(\bL,\bL)_{\rho_3}\otimes HF(\bL,\bL)_{\rho_3} \to & HF(\bL,\bL)_{\rho_0}.
\end{align*}
The multiplication is super-commutative by \eqref{eq:commutativity}.
\end{theorem}

%
%
%

\bibliographystyle{amsalpha}

\begin{thebibliography}{}
\bibitem{Dbrane} P. S. Aspinwall, T. Bridgeland, A. Craw, M. R. Douglas, M. Gross, A. Kapustin, G. W. Moore, G. Segal, B. Szendr\"oi and P. M. H. Wilson, {\em Dirichlet branes and mirror symmetry}, Clay Mathematics Institute Monograph, vol. 4 (2009), American Mathematical Society, Providence, RI.
\bibitem{BFK} M. Ballard, D. Favero and L. Katzarkov, {\em A category of kernels for matrix factorizations and its implications for Hodge theory}, Publ. Math. Inst. Hautes \'Etudes Sci. 120 (2014), 1-111.
\bibitem{BTW} A. Basalaev, A. Takahashi and E. Werner, {\em Orbifold Jacobian algebras for invertible polynomials}, arXiv:1608.08962.
\bibitem{CHL1} C.-H. Cho, H. Hong and S.-C. Lau, {\em Localized mirror functor for Lagrangian immersions, and homological mirror symmetry for $\mathbb{P}^1_{a,b,c}$}, J. Differential Geom. 106 (2017), no. 1, 45-126.
\bibitem{CHL2}  C.-H. Cho, H. Hong and S.-C. Lau, {\em Localized mirror functor constructed from a Lagrangian torus},  J. Geom. Phys. 136 (2019), 284-320.
\bibitem{CHKL}  C.-H. Cho, H. Hong, S.-H. Kim and S.-C. Lau, {\em Lagrangian Floer potential of orbifold spheres}, Adv. Math. 306 (2017), 344-426. 
\bibitem{CL} C.-H. Cho and S. Lee, {\em Kodaira-Spencer map, Lagrangian Floer theory and orbifold Jacobian algebras}, arXiv:2007.11732.
\bibitem{Dyc} T. Dyckerhoff, {\em Compact generators in categories of matrix factorizations}, Duke Math. J. 159 (2011), no. 2, 223-274.
\bibitem{FG} B. Fantechi and L. G\"ottsche, {\em Orbifold cohomology for global quotients}, Duke Math. J. 117 (2003), no. 2, 192-227.
\bibitem{HLL} W. He, S. Li and Y. Li, {\em G-twisted braces and orbifold Landau-Ginzburg models}, Comm. Math. Phys. 373 (2020), 175-217.
\bibitem{Kau} R. Kaufmann, {\em Orbifolding Frobenius algebras}, Int. J. Math. 14 (2003), no. 6, pp. 573-617.
\bibitem{Ki} T. Kimura, {\em Orbifold cohomology reloaded}, Toric topology, 231-240,
Contemp. Math., 460 (2008), Amer. Math. Soc., Providence, RI.
\bibitem{PP} A. Polishchuk and L. Positselski, {\em Hochschild (co)homology of the second kind I}, Trans. Amer. Math. Soc. 364 (2012), no. 10, 5311-5368.
\bibitem{PV} A. Polishchuk and A. Vaintrob, {\em Chern characters and Hirzebruch-Riemann-Roch formula for matrix factorizations}, Duke Math. J. 161 (2012), no. 10, 1863-1926.
\bibitem{Sei} P. Seidel, {\em Homological mirror symmetry for the genus two curve},  J. Algebraic Geom. 20 (2011), no. 4, 727-769.
\bibitem{Shk} D. Shklyarov, {\em On Hochschild invariants of Landau-Ginzburg orbifolds}, Adv. Theor. Math. Phys., 24(1) (2020), 189-258.
\bibitem{To} B. To\"en,  {\em The homotopy theory of dg-categories and derived Morita theory,} Invent. Math.
167 (2007), 615-667.
\bibitem{Tu} V. Turaev, {\em Homotopy field theory in dimension 2 and group algebras}, arxiv:math/9910010.
\end{thebibliography}

\end{document}